\documentclass[11pt]{amsart}
\usepackage[margin=1.2in]{geometry}
\usepackage{amscd,amsmath,amsxtra,amsthm,amssymb,stmaryrd,xr,mathrsfs,mathtools,enumerate,commath, comment}
\usepackage{stmaryrd}
\usepackage[dvipsnames]{xcolor}
\usepackage{commath}
\usepackage{dsfont}
\usepackage{comment}
\usepackage{tikz-cd}
\usepackage{longtable} % for 'longtable' environment
\usepackage{pdflscape} % for 'landscape' environment
\usepackage{booktabs}
\usepackage{hyperref}
\definecolor{vegasgold}{rgb}{0.77, 0.7, 0.35}
\definecolor{darkgoldenrod}{rgb}{0.72, 0.53, 0.04}
\definecolor{gold(metallic)}{rgb}{0.83, 0.69, 0.22}
\hypersetup{
 colorlinks=true,
 linkcolor=OliveGreen,
 citecolor=NavyBlue,
 }

\usepackage[all,cmtip]{xy}

\DeclareFontFamily{U}{wncy}{}
\DeclareFontShape{U}{wncy}{m}{n}{<->wncyr10}{}
\DeclareSymbolFont{mcy}{U}{wncy}{m}{n}
\DeclareMathSymbol{\Sh}{\mathord}{mcy}{"58}
\usepackage[T2A,T1]{fontenc}
\usepackage[OT2,T1]{fontenc}

\newtheorem{theorem}{Theorem}[section]
\newtheorem{Lemma}[theorem]{Lemma}
\newtheorem*{theorem*}{Theorem}
\newtheorem{corollary}[theorem]{Corollary}

\newtheorem{proposition}[theorem]{Proposition}

\newtheorem{definition}[theorem]{Definition}

\newtheorem*{ass*}{Assumption}

\newtheorem{remark}[theorem]{Remark}

\newcommand{\Z}{\mathbb{Z}}

\newcommand{\n}{\mathfrak{n}}

\newcommand{\defeq}{:=}

\newcommand{\RR}{\mathbb{R}}
\newcommand{\eps}{\epsilon}

\newcommand{\tM}{\Tilde{M}}

\newcommand{\op}[1]{\operatorname{#1}}

\numberwithin{equation}{section}

\begin{document}

\title[Integral Hasse principle for Markoff type cubic surfaces]{Integral Hasse principle for Markoff type cubic surfaces}

\author[]{Hrishabh Mishra\smallskip\smallskip\\ (\MakeLowercase{with an appendix by} Victor Y. Wang)}
\address[H. ~Mishra]{Universit\'e Paris Cit\'e – Campus des Grands Moulins, 5 Rue Thomas Mann, 75013 Paris, France}
\email{hrishabh@cmi.ac.in}

\address[V. Y. ~Wang]{IST Austria, Am Campus 1, 3400 Klosterneuburg, Austria}
\email{victor.wang@ist.ac.at}


\begin{abstract}
We establish new upper bounds on the number of failures of the integral Hasse principle within the family of Markoff type cubic surfaces $x^2+ y^2+ z^2- xyz= a$ with $\abs{a}\leq A$ as $A\to \infty$. Our bound improves upon existing work of Ghosh and Sarnak. As a result, we demonstrate that the integral Hasse principle holds for a density $1$ of surfaces in certain sparse sequences.
\end{abstract}

\maketitle
\vspace{-1.5em}
{
  \hypersetup{linkcolor=black}
  \tableofcontents
}
\vspace{-3em}
\section{Introduction}\label{introduction}
We consider the family of affine Markoff type cubic surfaces
\begin{equation}\label{markoff_family}
U_a: x^2+y^2+z^2-xyz=a,
\end{equation}
where $a \in \mathbb{Z}$. We begin by recalling some definitions. For each $a \in \mathbb{Z}$, let $\mathcal{U}_a$ denote the integral model of $U_a$ over $\mathbb{Z}$ defined by the above equation. Set $\mathbf{A}_{\mathbb{Z}} := \prod_{p \leq \infty}\mathbb{Z}_p$ with $\mathbb{Z}_{\infty} := \mathbb{R}$. We recall that $\mathcal{U}_a$ fails the \emph{integral Hasse principle} if $\mathcal{U}_a(\mathbf{A}_{\mathbb{Z}})$ is non-empty but $\mathcal{U}_a(\mathbb{Z})$ is empty. For $a \in \mathbb{Z}$, we say that $\mathcal{U}_a$ fails \emph{weak approximation} if the image of $U_a(\mathbb{Q})$ in $\prod_v U_a(\mathbb{Q}_v)$ is not dense. Furthermore, we say that $\mathcal{U}_a$ fails \emph{strong approximation} if $\mathcal{U}_a(\mathbb{Z})$ is not dense in $\pi_0(U_a(\mathbb{R})) \times \prod_p \mathcal{U}_a(\mathbb{Z}_p)$. Here, $\pi_0$ denotes the set of connected components. Failures of all three properties in the family \eqref{markoff_family} have been extensively studied in recent years.

\par In \cite{ghosh2017integral}, Ghosh and Sarnak investigated the failures of the integral Hasse principle for the family \eqref{markoff_family}. From \cite[Proposition 6.1]{ghosh2017integral}, it follows that $\mathcal{U}_a(\textbf{A}_{\Z}) \neq \emptyset$ if and only if $a \not\equiv 3 \pmod{4}$ and $a \not\equiv \pm 3 \pmod{9}$. Values of $a$ that satisfy these conditions are called \emph{admissible}. Therefore, the set of admissible integers $a \in \Z$ has a natural density of $7/12$. Let $\mathcal{E}$ denote the set of integers $a \in \Z$ for which $\mathcal{U}_a$ fails the integral Hasse principle. For $A > 0$, we define
\[
\mathcal{E}(A):= \{a \in \mathcal{E}: |a|\leq A\}.
\]
It follows from their work that $\#\mathcal{E}(A)=o(A)$ as $A \to \infty$. A lower bound on the number of failures of the integral Hasse principle
\[
\#\mathcal{E}(A)\gg \sqrt{A}(\log A)^{-1/2}.
\]
is given by \cite[Theorem 1.2 (i)]{ghosh2017integral}. They obtain these counterexamples based on the quadratic-reciprocity law. Ghosh and Sarnak also provide numerical computations in \cite[\S10]{ghosh2017integral} indicating that $\#\mathcal{E}(A)$ is at least $A^\theta$ for some $1/2 < \theta < 1$.
\par The problem of determining asymptotics for the failures of the integral Hasse principle for Markoff type surfaces remains largely unresolved. In this paper, we aim to improve the upper bound on the size of the set $\mathcal{E}(A)$ and to analyze the failures of the integral Hasse principle in sparse sequences. We will now state our main result and present a straightforward corollary.
\begin{theorem}\label{main-thm}
    We have that
    \[
    \#\mathcal{E}(A) \ll_\eps \frac{A}{(\log A)^{2-\eps}},
    \]
    for all $A\geq 2$.
\end{theorem}
\noindent We note that \cite[Theorem 1.2(ii)]{ghosh2017integral}, if fully quantified, would likely yield a saving with a power of $\log\log{A}$ rather than a power of $\log{A}$. Let $c$ be an integer. We consider the following subfamily of Markoff type surfaces
\begin{equation}\label{prime_markoff_family}
    U_{p+c}: x^2+y^2+z^2-xyz=p+c,
\end{equation}
where $p \in \Z$ is a rational prime. We have the following corollary.
\begin{corollary}
    For every integer $c$, the integral Hasse principle almost always holds in the family \eqref{prime_markoff_family}.
\end{corollary}

\begin{remark}
    Similar results can be deduced for several other interesting sequences. For example, we can restrict the primes in the family \eqref{prime_markoff_family}. Using \cite[Theorem 1]{iwaniec73}, we can restrict primes in the family \eqref{prime_markoff_family} to values of suitable quadratic polynomials in two variables. For instance, we can restrict to prime values of the polynomial $P(s, t)= s^2 + t^2 + 1$. 
\end{remark}

\par Loughran and Mitankin also studied the family \eqref{markoff_family} in \cite{loughran2021integral}, focusing on weak and strong approximation in this family \cite[Theorems 1.1 and 1.2]{loughran2021integral}. They prove that weak and strong approximation almost always fails for Markoff type surfaces. Simultaneously, using reduction theory, Colliot-Th\'el\`ene, Wei, and Xu demonstrated in \cite{colliot2020brauer} that strong approximation always fails.

\par Failures of the integral Hasse principle have also been studied in \cite{loughran2021integral} and \cite{colliot2020brauer} using the integral Brauer-Manin obstruction. From \cite[Theorem 1.4]{loughran2021integral}, we know that the number of integers $\abs{a} \leq A$ for which $\mathcal{U}_a$ has an integral Brauer-Manin obstruction to the integral Hasse principle is $\asymp \sqrt{A}(\log A)^{-1/2}$. Theorem 1.5 in \cite{loughran2021integral} provides an asymptotic lower bound on the number of failures of the integral Hasse principle not explained by the Brauer-Manin obstruction. Specifically, they show that the number of such failures $a \in \Z$ with $\abs{a} \leq A$ is at least $c \sqrt{A} (\log A)^{-1}$ for some $c > 0$. Colliot-Thélène, Wei, and Xu improve this lower bound to $c \sqrt{A} (\log A)^{-1/2}$ \cite[Theorem 5.8]{colliot2020brauer}. They also prove that the failures of the integral Hasse principle discussed in \cite{ghosh2017integral} are all of Brauer-Manin type. An improvement in the power of $\log A$ in the denominator of the lower bound together with the bounds on the number of surfaces with an integral Brauer-Manin obstruction to the integral Hasse principle will demonstrate that almost all failures of the integral Hasse principle in the family \eqref{markoff_family} are not explained by the Brauer-Manin obstruction, as suggested by \cite[Conjecture 10.2]{ghosh2017integral}.

\par The strategy for proving \cite[Theorem 1.2(ii)]{ghosh2017integral} involves deriving an upper bound for a suitable approximate variance and combining this with an appropriate small value estimate to establish the desired upper bound on the number of failures of the integral Hasse principle. A similar method is used in \cite{wang2024prime} to obtain bounds comparable to our main result for the sum of three cubes, conditionally. Our approach integrates these two methods to prove our main result. In doing so, we employ additive approximants for the singular series, as done in \cite{wang2024prime}, unlike the multiplicative approximants used in \cite{ghosh2017integral}. However, we note that since the Markoff polynomial is not homogeneous, our analysis of the singular series and the real densities is somewhat more involved.

\par In \cite[\S9]{ghosh2017integral}, bounding the approximate variance with suitable parameters, after fixing one of the variables in \eqref{markoff_family}, reduces to bounding both a diagonal sum and a non-diagonal sum. Ghosh and Sarnak address the diagonal sum using the methods from \cite{blomer2006} and the non-diagonal sum using the version of the circle method described in \cite{heath1996new}. A similar situation arises in our work: while the diagonal sum is relatively straightforward and is estimated using elementary methods, the non-diagonal sum requires uniformity in various aspects that is not explicitly addressed in \cite{heath1996new}. The required estimate is established in the appendix at the end of the paper.

\par We conclude the Introduction with a remark that our main result is compatible with the results in \cite{gamburdetal2019} and Conjecture 1.1 in \cite{browning2024integralpointscubicsurfaces}. Additionally, our proof leads to a better lower bound for the number of orbits $h_M(a)$ for almost all $a \in \Z$; see \cite[\S1]{ghosh2017integral} for relevant definitions and more details.

\begin{remark}
    It might be possible to adapt these methods to prove more restricted versions of the Theorem \ref{main-thm}, for example, by taking one of the variables to be prime or under similar conditions.
\end{remark}

\subsection{Structure of the paper}
Including the Introduction, this paper contains five sections and an appendix at the end. Section \ref{preliminaries} develops the framework for the proof of the main theorem and defines analogs of several quantities from \cite{wang2024prime} within our setting. It also includes results on local and real densities. In Section \ref{variance-analysis}, we prove an upper bound on the approximate variance defined in Section \ref{preliminaries}. Section \ref{small_est} is dedicated to establishing a small value estimate, following the approach in \cite[\S3]{wang2024prime}. In Section \ref{last section}, we prove a lower bound on a real density and combine results from Sections \ref{preliminaries}, \ref{variance-analysis}, and \ref{small_est} to complete the proof of Theorem \ref{main-thm}. The appendix contains the estimates with the desired uniformity required to control the non-diagonal sum in Section \ref{variance-analysis}.

\subsection{Notation}
Throughout the paper, we set $M \in \Z[x, y, z]$ to be the \emph{Markoff polynomial} given by $M(x, y, z) := x^2 + y^2 + z^2 - xyz$. We define $\tM(\mathbf{x}, \mathbf{y}) := M(\mathbf{x}) - M(\mathbf{y}) \in \Z[\mathbf{x}, \mathbf{y}]$. For integers $m$ and $n$, we write $m \mid n^\infty$ if there exists a positive integer $k$ such that $m \mid n^k$.

\subsection{Acknowledgements}
The author would like to express special thanks to Victor Y. Wang for suggesting this problem, engaging in numerous insightful discussions, providing valuable suggestions, and carefully reviewing earlier drafts. Thanks are also due to Tim Browning and Peter Sarnak for their helpful suggestions, comments, and interest in the paper. The author is particularly grateful to the Browning group and the Institute of Science and Technology, Austria, for their hospitality during the time this article was written.

\section{Preliminaries}\label{preliminaries}
In this section, we define an approximate variance and describe the framework for the proof of Theorem \ref{main-thm}. We prove analogs of various results on local and real densities from \cite[\S2]{wang2024prime}. These results are crucial in proving Theorem \ref{variance_bound} in Section \ref{variance-analysis}.

\subsection{The approximate variance}
Let $0<\eta\ll \delta\ll 1$. Let $\nu\ge 0$ be a smooth bump function of mass $1$ supported on $[1,2] \subset \RR$. Let $\varsigma\in \{\pm 1\}$. For $B\ge 1$ and $(x,y,z)\in \RR^3$, let
\begin{equation}\label{def-nu}
\nu_{B,\varsigma}(x,y,z) \defeq \int_{B^{\delta+\eta}}^{B^{\delta+2\eta}} \nu{\left(\frac{\varsigma z}{Z}\right)}
\int_{B^{1-\delta}}^{B^{1-\delta+\eta}} \nu{\left(\frac{y}{Y}\right)}
\nu{\left(\frac{x}{X}\right)}
\,\frac{dY}{Y}\frac{dZ}{Z},
\end{equation}
where we write $X=\frac{B^2}{YZ}$ for convenience. Note that $\nu_{B,\varsigma}$ is supported on
\begin{equation}\label{basic-support-conditions}
\begin{split}
    \varsigma z=\abs{z}\asymp & Z\in [B^{\delta+\eta},B^{\delta+2\eta}],
\quad y\asymp Y\in [B^{1-\delta},B^{1-\delta+\eta}],\\
&\quad x\asymp X\in [B^{1-3\eta},B^{1-\eta}].
\end{split}
\end{equation}
In particular, we have $xyz \asymp B^2$ and $x^2,y^2,z^2 \ll B^{2-2\eta}$. Moreover, the inequality
\begin{equation*}
B^{\delta-4\eta} Y
\le B^{1-3\eta}
\le X
\le B^{1-\eta}
\le B^{-2\eta} YZ
\end{equation*}
always holds. Next, we define several quantities analogous to those in \cite{wang2024prime}.
\begin{definition}\label{real-den}
    For $a\in \RR$ and $B\gg 1$, let
\begin{equation*}
\begin{split}
\sigma_{\infty,a,\nu}(B,\varsigma)
&\defeq \lim_{\eps\to 0}{(2\eps)^{-1}
\int_{\substack{(x,y,z)\in \RR^3 \\ \abs{M_a(x,y,z)}\le \eps}}
\nu_{B,\varsigma}(x,y,z)\, dx\,dy\,dz} \\
&= \int_{(x,y)\in \RR^2} \frac{\nu_{B,\varsigma}(x,y,z)}{\abs{2z-xy}}
\Big{\vert}_{M(x,y,z) = a}\, dx\,dy.
% &= \int_{(y,z)\in \RR^2} \frac{\nu_{B,\varsigma}(x,y,z)}{\abs{2x-yz}}
% \Big{\vert}_{M(x,y,z) = a}\, dy\,dz.
\end{split}
\end{equation*}
\end{definition}
\noindent
The last equality follows from \cite[\S5.4, par. 4]{chambert2010igusa}. Note that $\abs{2z-xy} \asymp XY \ge B^{2-\delta-3\eta}$ by \eqref{basic-support-conditions}.
% Note that $\abs{2x-yz} \asymp YZ \ge B^{1+\eta}$ by \eqref{basic-support-conditions}.
For $a\in \Z$, let $M_a:= M- a$. For $B\ge 1$, we also define
\begin{equation}\label{definition-r_a}
r_a(B,\varsigma)
\defeq \sum_{\substack{x,y,z\in \Z \\ M(x,y,z) = a}}
\nu_{B,\varsigma}(x,y,z).
\end{equation}
\noindent
For $m \in \Z_{\geq 1}$ set
\begin{equation}\label{definition-T}
T_a(m) \defeq \sum_{\substack{1\le u\le m \\ \gcd(u,m)=1}}
\sum_{1\le x,y,z\le m} e_m(u M_a(x,y,z)).
\end{equation}
\noindent
For $K\geq 1$, let
\begin{equation}\label{definition_s_a}
s_a(K) \defeq \sum_{m\leq K} m^{-3} T_a(m).
\end{equation}
\noindent
Now, we define the analog of the approximate variance \cite[Eq. 2.5, pp. 5]{wang2024prime}. For $K\ge 1$ and $B\gg 1$, let
\begin{equation}\label{var-def}
\op{Var}(B,\varsigma,K)\defeq \sum_{a\in \Z} (r_a(B,\varsigma) - s_a(K)\sigma_{\infty,a,\nu}(B,\varsigma))^2.
\end{equation}

\subsection{Local Densities} Given an integer $q\geq 1$, we consider the following quantity
\begin{equation*}
    N(q):= q^{-5}\#\{\textbf{x}\in (\Z/q\Z)^6: \tM(\textbf{x})\equiv 0 \bmod q\}.
\end{equation*}

For $m\geq 1$, we also define an analog of the quantity $S_{\textbf{0}}$ from \cite[Eq. 2.8, pp. 5]{wang2024prime} as below
\begin{equation}\label{definition-S}
    S_{\textbf{0}}(m)\defeq \sum_{u\in (\Z/m\Z)^{\times}}\sum_{\textbf{x}\in (\Z/m\Z)^6}e_m(u\tM(\textbf{x})).
\end{equation}
Rewriting $N(q)$ as
\[
\sum_{a \in \Z/q\Z}q^{-6}\sum_{\textbf{x}\in (\Z/q\Z)^6}e_q(a\tM(\textbf{x})),
\]
we deduce that
\[
N(q)= \sum_{m\mid q} m^{-6}S_{\textbf{0}}(m).
\]
We also define the related sum obtained using the polynomial
$$G_{a_1, a_2}(\textbf{v}, \textbf{w}):= v_1^2+v_2^2-a_1v_1v_2-w_1^2-w_2^2+a_2w_1w_2+a_1^2-a_2^2$$
for a fixed pair $(a_1, a_2) \in \Z^2$. Note that $G_{a_1, a_2}(\textbf{v}, \textbf{w})= \tM(\textbf{v}, a_1, \textbf{w}, a_2)$. For $m \geq 1$, consider
\begin{equation}\label{definition-S-fixed}
    S_{\textbf{0}}(a_1, a_2; m):= \sum_{u\in (\Z/m\Z)^\times}\sum_{(\textbf{v}, \textbf{w}) \in (\Z/{m}\Z)^4} e_{m}(uG_{a_1, a_2}(\textbf{v},\textbf{w}))
\end{equation}
It follows from \cite[Lemma 2.13]{browningcubic} that for any $a_1, a_2 \in \Z^2$, the sum $S_{\textbf{0}}(a_1, a_2; -)$ is multiplicative. We now prove some important bounds using the local computations in \cite[Appendix B]{ghosh2017integral}. The next two lemmas bound the sums $S_{\textbf{0}}(a_1, a_2; m)$ defined above.
\begin{Lemma}\label{fix_a_1_2_sum_p_odd}
  Let $a_1, a_2$ be integers with $a_1\neq \pm a_2$ and $\abs{a_1}, \abs{a_2} > 2$. For $p\geq 3$ and $\ell \geq 2$ we have the following bounds 
  \begin{equation*}
      p^{-4\ell}\abs{S_{\textbf{0}}(a_1, a_2; p^{\ell})} \leq
      \begin{cases}
          1 \text{ if } p^{\ell-1}\mid \gcd(a_1^2-4, a_2^2-4),\\
          p^{-\ell} \text{ if } p^{\ell-1}\mid (a_1^2-a_2^2)\text{ and }p \nmid (a_1^2-4)(a_2^2-4),\\
          0 \text{ otherwise}.
      \end{cases}
  \end{equation*}
  We also have the following bounds in the case $\ell=1$.
  \begin{equation*}
      p^{-4}\abs{S_{\textbf{0}}(a_1, a_2; p)} \leq
      \begin{cases}
          1 \text{ if } p\mid \gcd(a_1^2-4, a_2^2-4),\\
          p^{-1} \text{ if } p\mid a_i^2-4 \text{ and } p\nmid a_j^2-4, 1\leq i\neq j \leq 2,\\
          p^{-1} \text{ if } p\mid (a_1^2-a_2^2)\text{ and }p \nmid (a_1^2-4)(a_2^2-4),\\
          p^{-2} \text{ otherwise}.
      \end{cases}
  \end{equation*}
\end{Lemma}
\begin{proof}
    Immediate from \cite[Proposition B.4]{ghosh2017integral}.
\end{proof}
Next, we state the analog of the Lemma \ref{fix_a_1_2_sum_p_odd} for $p=2$.
\begin{Lemma}\label{fix_a_1_a_2_sum_p_two}
    Let $a_1, a_2$ be integers with $a_1\neq \pm a_2$ and $\abs{a_1}, \abs{a_2} > 2$. For $\ell \geq 10$ we have the following bounds
    \begin{equation*}
        2^{-4\ell}S_{\textbf{0}}(a_1, a_2; 2^{\ell}) \ll
        \begin{cases}
           \min\left(1, \frac{\gcd(a_1^2-4, a_2^2-4)}{2^{\ell}}\right) & \text{if } 2^{\ell-5} \mid (a_1^2-4)(a_2^2-4)(a_1^2-a_2^2),\\
           2^{-2\ell}& \text{otherwise}.
        \end{cases}
    \end{equation*}  
\end{Lemma}
\begin{proof}
    The lemma follows from \cite[Proposition B.14]{ghosh2017integral}.
\end{proof}
We combine the above two lemmas and use multiplicativity to prove an important upper bound for the truncated singular series. We must have uniformity in $a_1, a_2$ for our application in Section \ref{variance-analysis}.
\begin{proposition}\label{uniform_singular}
    We have the bound
    \[
    \sum_{m> K}m^{-4}\abs{S_{\textbf{0}}(a_1, a_2; m)} \ll_{\eps} |(a_1^2-4)(a_2^2-4)(a_1^2-a_2^2)|^\eps\min(1, \gcd(a_1^2-4, a_2^2-4)^5/K),
    \]
    for all $K\geq 1$ and $a_1\neq \pm a_2$ and $\abs{a_1}, \abs{a_2} > 2$.
\end{proposition}
\begin{proof}
    Suppose $m > K$, we define the sets isolating odd prime factors of $m$
    \begin{enumerate}
        \item $\mathfrak{P}_1:= \{p^{\ell}: p \text{ is an odd prime }, \ell\geq 1, p^{\ell}\parallel m, p^{\max(\ell-1, 1)}\mid \gcd(a_1^2-4, a_2^2-4)\}$,
        \item $\mathfrak{P}_2:= \{p^{\ell}: p \text{ is an odd prime }, \ell\geq 1, p^{\ell}\parallel m, p^{\max(\ell-1, 1)}\mid (a_1^2- a_2^2) \text{ and }p \nmid (a_1^2-4)(a_2^2-4)\}$,
        \item $\mathfrak{P}_3:= \{p: p \text{ is an odd prime }, p\parallel m, p\mid a_i^2-4 \text{ and } p\nmid a_j^2-4, 1\leq i\neq j \leq 2\}$,
    \end{enumerate}
    We also isolate powers of the prime $2$, we consider the set
    \[
    \mathfrak{Q}:=\{2^\ell: 2^\ell\parallel m\text{ and }\ell<10\text{ or }2^{\ell-5} \mid (a_1^2-4)(a_2^2-4)(a_1^2-a_2^2)\}.
    \]
    Now, we define
    \[
    m_1:= \prod_{p^\ell\in \mathfrak{P}_2\cup\mathfrak{P}_3}p^\ell,\, m_2:= \prod_{p^\ell\in \mathfrak{P}_1}p^\ell, \text{ and }m_3:=\prod_{p^\ell\in \mathfrak{Q}}p^\ell.
    \]
    Using Lemma \ref{fix_a_1_2_sum_p_odd} and Lemma \ref{fix_a_1_a_2_sum_p_two}, we deduce that
    \[
    m^{-4}\abs{S_{\textbf{0}}(a_1, a_2; m)} \ll \frac{m_1 m_2^2m_3\gcd(a_1^2-4, a_2^2-4)}{m^2}.
    \]
    Then we have that
    \[
     \sum_{m > K}m^{-4}\abs{S_{\textbf{0}}(a_1, a_2; m)} \ll \sum_{m > K}  \frac{m_1 m_3\gcd(a_1^2-4, a_2^2-4)^5}{m^2}.
    \]
    For a fixed $m_1, m_3$, we have the bound
    \[
    \sum_{m> K}\frac{m_1 m_3\gcd(a_1^2-4, a_2^2-4)^5}{m^2} \ll \frac{\gcd(a_1^2-4, a_2^2-4)^5}{K}.
    \]
    Since there are at most $O_\eps(|(a_1^2-4)(a_2^2-4)(a_1^2-a_2^2)|^\eps)$ choices for $m_1, m_2$, and $m_3$ we deduce that
    \[
     \sum_{m > K}m^{-4}\abs{S_{\textbf{0}}(a_1, a_2; m)} \ll |(a_1^2-4)(a_2^2-4)(a_1^2-a_2^2)|^\eps \frac{\gcd(a_1^2-4, a_2^2-4)^5}{K}.
    \]
    We note again that Lemma \ref{fix_a_1_2_sum_p_odd} and Lemma \ref{fix_a_1_a_2_sum_p_two} give us the bound
    \[
    m^{-4}\abs{S_{\textbf{0}}(a_1, a_2; m)} \ll \frac{m_1 m_2^2m_3^2}{m^2},
    \]
    let $\tilde{m}:= m/m_1m_2m_3$, then we have the bound
    \[
     \sum_{m > K}m^{-4}\abs{S_{\textbf{0}}(a_1, a_2; m)} \ll \sum_{m>K}\frac{1}{\tilde{m}^2}.
    \]
    Given any $n\geq 1$, there are at most $O_\eps(|(a_1^2-4)(a_2^2-4)(a_1^2-a_2^2)|^\eps)$ integers $m\geq 1$ such that $\tilde{m}=n$, hence we have the bound
    \[
    \sum_{m > K}m^{-4}\abs{S_{\textbf{0}}(a_1, a_2; m)} \ll_\eps |(a_1^2-4)(a_2^2-4)(a_1^2-a_2^2)|^\eps.
    \]
    The lemma is proved.
\end{proof}

We end this subsection by establishing identities involving $T_a(-)$ and $S_{\textbf{0}}(-)$ defined in \eqref{definition-T} and \eqref{definition-S}.
\begin{Lemma}\label{id-t-s}
    Let $m\geq 1$, then the following hold
  \begin{equation}\label{id-t-s-1}
      \frac{1}{m}\sum_{b \in \Z/m\Z}T^2_b(m)= S_{\textbf{0}}(m),
  \end{equation}
  \begin{equation}\label{id-t-s-2}
      \sum_{\textbf{x} \in (\Z/m\Z)^3}T_{M(\textbf{x})}(m)= S_{\textbf{0}}(m).
  \end{equation}
\end{Lemma}
\begin{proof}
        From the definition for $T_a(-)$ in \eqref{definition-T} we have that
    \[
    T_b^2(m)=\sum_{\substack{1\le u, v\le m \\ \gcd(u,m), \gcd(v, m)=1}}
\sum_{\substack{\textbf{y}\in (\Z/m\Z)^3,\\ \textbf{z}\in (\Z/m\Z)^3}} e_m(u M_b(\textbf{y})+v M_b(\textbf{z})),
    \]
    for all $b \in \Z/m\Z$. Summing over $b \in \Z/m\Z$, changing the order of summation and noting that $\sum_{b\in \Z/m\Z}e_m(-(u+v)b) = m\cdot\mathds{1}_{m\mid u+v}$ we obtain the following
    \[
    \sum_{b \in \Z/m\Z}T^2_b(m)= m\cdot \sum_{u\in (\Z/m\Z)^{\times}}\sum_{\textbf{x}\in (\Z/m\Z)^6}e_m(u\tM(\textbf{x})).
    \]
    This proves \eqref{id-t-s-1}. Next, again using the definition \eqref{definition-T}
    \[
    \sum_{\textbf{x} \in (\Z/m\Z)^3}T_{M(\textbf{x})}(m)= \sum_{\textbf{x} \in (\Z/m\Z)^3}\sum_{\substack{1\le u\le m \\ \gcd(u,m)=1}}
\sum_{\textbf{y}\in (\Z/m\Z)^3} e_m(u (M(\textbf{y})-M(\textbf{x}))),
    \]
    changing the order of summation proves \eqref{id-t-s-2}.
\end{proof}
The next lemma will imply Proposition \ref{imp-sum} below.
\begin{Lemma}\label{sum-zero}
    Let $n_1, n_2 \geq 1$ be such that $n_1\neq n_2$, then
    \[
    \sum_{b \in \Z/n_1n_2\Z}T_b(n_1)T_b(n_2)= 0.
    \]
\end{Lemma}
\begin{proof}
    Suppose $n_1 < n_2$. Then, for all pairs $(u,v) \in (\Z/n_1\Z)^\times \times (\Z/n_2\Z)^\times$ we have
    \[
    \sum_{b\in \Z/n_1n_1\Z} e_{n_1n_2}(-(un_2+vn_1)b)=0.
    \]
    Using the definition in \eqref{definition-T} and the above equality, we conclude the lemma. The sum in the lemma is
    \[
    \sum_{\substack{u \in (\Z/n_1\Z)^\times,\\ v \in (\Z/n_2\Z)^\times}}\sum_{\substack{\textbf{x}\in (\Z/m\Z)^3,\\ \textbf{y}\in (\Z/m\Z)^3}}e_{n_1n_2}(un_2M(\textbf{x})+vn_1M(\textbf{y}))\sum_{b \in \Z/n_1n_2\Z}e_{n_1n_2}(-(un_2+vn_1)b).
    \]
    The innermost sum is always zero, hence the entire sum is zero.
\end{proof}
The following Proposition is a direct consequence of the above Lemmas \ref{id-t-s} and \ref{sum-zero}. We will use this result for analyzing the approximate variance in Section \ref{variance-analysis}.
\begin{proposition}\label{imp-sum}
    Let $K\geq 1$, then
    \begin{equation*}
        \sum_{n_1, n_2 \leq K}\frac{1}{n_1n_2}\sum_{b\in \Z/n_1n_2\Z}\frac{T_b(n_1)T_b(n_2)}{(n_1n_2)^3}=\sum_{m\leq K} \frac{S_{\textbf{0}}(m)}{m^6},
    \end{equation*}
    and
    \begin{equation*}
        \sum_{n\leq K}\frac{1}{n^3}\sum_{\textbf{x} \in (\Z/m\Z)^3}\frac{T_{M(\textbf{x})}(m)}{n^3}=\sum_{m\leq K} \frac{S_{\textbf{0}}(m)}{m^6}.
    \end{equation*}
\end{proposition}
\subsection{Real Densities} We estimate certain sums involving the real density defined in \ref{real-den}. Before stating these results we define the following related densities for $B \gg 1$.
\begin{equation}\label{moment}
    \begin{split}
\sigma^{\otimes 2}_{\infty,\nu}(B,\varsigma)
&\defeq \lim_{\eps\to 0}{(2\eps)^{-1}
\int_{\substack{(\textbf{v},x,y,z)\in \RR^6 \\ \abs{\tM(\textbf{v},x,y,z)}\le \eps}}
\nu_{B,\varsigma}(\textbf{v})\nu_{B,\varsigma}(x,y,z)\, d\textbf{v}\,dx\,dy\,dz} \\
&= \int_{(\textbf{v},x,y)\in \RR^5} \frac{\nu_{B,\varsigma}(\textbf{v})\nu_{B,\varsigma}(x,y,z)}{\abs{2z-xy}}
\Big{\vert}_{\tM(\textbf{v},x,y,z) = 0}\, d\textbf{v}\,dx\,dy.
\end{split}
\end{equation}
Let us fix integers $a_1, a_2$, we also define
\begin{equation}\label{moment_a_1_2}
    \begin{split}
\sigma^{\otimes 2}_{\infty,\nu}(a_1, a_2, B,\varsigma)
&\defeq \lim_{\eps\to 0}{(2\eps)^{-1}
\int_{\substack{(v_1, v_2, \textbf{w})\in \RR^4 \\ \abs{G_{a_1, a_2}}\le \eps}}
\nu_{B,\varsigma}(\textbf{w}, a_2)\nu_{B,\varsigma}(v_1,v_2,a_1)\, d\textbf{w}\,dv_1\,dv_2} \\
&= \int_{(v_2,\textbf{w})\in \RR^3} \frac{\nu_{B,\varsigma}(\textbf{w}, a_2)\nu_{B,\varsigma}(v_1,v_2,a_1)}{\abs{2v_1-v_2a_1}}
\Big{\vert}_{G_{a_1,a_2} = 0}\, d\textbf{w}\,dv_2.
\end{split}
\end{equation}
The equalities in both definitions are due to \cite[\S5.4, par. 4]{chambert2010igusa}. First, we prove an upper bound on the above real density. The proof is straightforward but this bound is crucial for the next section.
\begin{Lemma}\label{bound_real_a_1_2_lemma}
    We have that
    \begin{equation}\label{bound_real_a_1_2}
    \sigma^{\otimes 2}_{\infty,\nu}(a_1, a_2, B,\varsigma) \ll \frac{B^2(\log B)^2}{\abs{a_1a_2}}.
\end{equation}
\end{Lemma}
\begin{proof}
    We note that $\abs{2v_1-v_2a_1} \gg v_2a_1$, therefore
    \begin{equation*}
        \begin{split}
            \int_{(v_2,\textbf{w})\in \RR^3} &\frac{\nu_{B,\varsigma}(\textbf{w}, a_2)\nu_{B,\varsigma}(v_1, v_2,a_1)}{\abs{2v_1-v_2a_1}}
\Big{\vert}_{G_{a_1,a_2} = 0}\, d\textbf{w}\,dv_2 \\
& \ll \int_{(v_2, \textbf{w})\in \RR^3} \frac{\nu_{B,\varsigma}(\textbf{w}, a_2)\nu_{B,\varsigma}(v_1,v_2,a_1)}{\abs{v_2a_1}}
\Big{\vert}_{G_{a_1,a_2} = 0}\, d\textbf{w}\,dv_2.
        \end{split}
    \end{equation*}
    Further, we also note the last integral is at most
    \[
    \left(\int_{\textbf{w}\in \RR^2}\nu_{B,\varsigma}(w_1, w_2, a_2)\, d\textbf{w}\right)\left(\int_{v_2\in \RR} \frac{\sup_{v_1\in \RR}\nu_{B,\varsigma}(v_1,v_2,a_1)}{|v_2a_1|}\,dv_2\right).
    \]
    We bound both integrals separately. Interchanging the order of integrals we deduce that
    \begin{equation*}
        \begin{split}
            \int_{\textbf{w}\in \RR^2}\nu_{B,\varsigma}(w_1, w_2, a_2)\, d\textbf{w}&=\int_{B^{\delta+\eta}}^{B^{\delta+2\eta}}\int_{B^{1-\delta}}^{B^{1-\delta+\eta}}\int_{\textbf{w}\in \RR^2} \nu{\left(\frac{\varsigma a_2}{Z}\right)}\nu{\left(\frac{w_2}{Y}\right)}
\nu{\left(\frac{w_1}{X}\right)}\,d\textbf{w}\,\frac{dY}{Y}\frac{dZ}{Z}\\
&\ll \int_{B^{\delta+\eta}}^{B^{\delta+2\eta}}\int_{B^{1-\delta}}^{B^{1-\delta+\eta}}\mathds{1}_{|a_2|/2\leq Z \leq a_2} XY\,\frac{dY}{Y}\frac{dZ}{Z}\\
&\ll \frac{B^2\log B}{|a_2|},
        \end{split}
    \end{equation*}
  the last bound follows from the asymptotics
  \[
  XY \asymp \frac{B^2}{Z} \asymp \frac{B^2}{\abs{a_2}}.
  \]
  Similarly, for the second integral, we deduce that
\begin{equation*}
    \begin{split}
        \int_{v_2\in \RR}\frac{\sup_{v_1\in \RR}\nu_{B,\varsigma}(v_1,v_2,a_1)}{|v_2a_1|}\,dv_2&\ll \int_{B^{\delta+\eta}}^{B^{\delta+2\eta}}\int_{B^{1-\delta}}^{B^{1-\delta+\eta}}\int_{v_2\in \RR} \nu{\left(\frac{\varsigma a_1}{Z}\right)}\nu{\left(\frac{v_2}{Y}\right)}\,\frac{dv_2}{\abs{v_2a_1}}\,\frac{dY}{Y}\frac{dZ}{Z}\\
& \ll \int_{B^{\delta+\eta}}^{B^{\delta+2\eta}}\int_{B^{1-\delta}}^{B^{1-\delta+\eta}}\mathds{1}_{|a_1|/2\leq Z \leq a_1}\,\frac{dY}{Y\abs{a_1}}\frac{dZ}{Z}\\
&\ll \frac{(\log B)}{|a_1|}.
    \end{split}
\end{equation*}
This concludes the proof of the lemma.
\end{proof}
For $\sigma^{\otimes 2}_{\infty,\nu}(B,\varsigma)$ we have the following result in the Archimedean setting instead of Proposition \ref{imp-sum}.
\begin{Lemma}\label{int-eq}
    We have that,
    \begin{equation*}
        \int_{a\in \RR}\sigma_{\infty,a,\nu}(B,\varsigma)^2\,da= \sigma^{\otimes 2}_{\infty,\nu}(B,\varsigma),
    \end{equation*}
    and
    \begin{equation*}
        \int_{\textbf{x}\in \RR^3}\nu_{B,\varsigma}(\textbf{x})\sigma_{\infty,M(\textbf{x}),\nu}(B,\varsigma)\, d\textbf{x}= \sigma^{\otimes 2}_{\infty,\nu}(B,\varsigma).
    \end{equation*}
\end{Lemma}
\begin{proof}
    Using the definition in \eqref{real-den}, changing the order of integrals and using the relation $M(x_2, y_2, z_2)=a$, we obtain that
    \begin{equation*}
        \begin{split}
            \int_{a\in \RR}&\sigma_{\infty,a,\nu}(B,\varsigma)^2\,da\\
            &= \int_{\RR^4}\int_{a \in \RR}\left.\frac{\nu_{B,\varsigma}(v_1,v_2,v_3)\nu_{B,\varsigma}(x_1,y_1,z_1)}{\abs{2v_3-v_1v_2}\abs{2z-xy}}
\right\vert_{\substack{M(\textbf{v}) = a,\\ M(x, y, z)=a}}\,da\, \,dv_1\,dv_2\,dx\,dy\\
&= \int_{\RR^4}\int_{v_3 \in \RR}\left.\frac{\nu_{B,\varsigma}(v_1,v_2,v_3)\nu_{B,\varsigma}(x_1,y_1,z_1)}{\abs{2z-xy}}
\right\vert_{\substack{\tM=0}}\,dv_3\,dv_2\,dv_1\,dx\,dy.
        \end{split}
    \end{equation*}
    This proves the first identity. For the second one, we again expand using the definition \eqref{real-den} to obtain
    \[
    \int_{\textbf{v}\in \RR^3}\nu_{B,\varsigma}(\textbf{v})\sigma_{\infty,M(\textbf{v}),\nu}(B,\varsigma)\, d\textbf{v}= \int_{(\textbf{v}, x, y)\in \RR^5}\left.\frac{\nu_{B,\varsigma}(\textbf{v})\nu_{B,\varsigma}(x, y, z)}{\abs{2z-xy}}
\right\vert_{\substack{\tM(\textbf{v}, x, y, z)=0}}\,dx\,dy\,d\textbf{v}.
    \]
    This completes the proof of the lemma.
\end{proof}
Next, we prove some crucial derivative bounds. We use these bounds to prove Proposition \ref{real_sum}.
\begin{Lemma}\label{der_bound}
    Suppose $0 < \eta \ll \delta \ll 1$, then we have following bounds
    \begin{enumerate}
        \item For $n> 2$ we have that
        \begin{equation}\label{der1}
            \partial_a^n \sigma_{\infty,a,\nu}(B,\varsigma)\ll_n \frac{(\log B)^2}{B^{(2-\delta-3\eta)(n-2)}}.
        \end{equation}
        \item For $n> 2$ and $i=1,2,3$ we have that
        \begin{equation}\label{der2}
            \partial_{z_i}^n\nu_{B,\varsigma}(\textbf{z})\sigma_{\infty,M(\textbf{z}),\nu}(B,\varsigma) \ll_n \frac{B^2(\log B)^4}{B^{(\delta-\eta)n}}.
        \end{equation}
        \item For $n > 2$ and $i=1,2$ we have that
        \begin{equation}\label{der3}
            \partial_{a_i}^n \sigma_{\infty, \nu}^{\otimes 2}(a_1, a_2)(B, \varsigma) \ll_n \frac{B^3(\log B)^4}{B^{(\delta-\eta) n}}.
        \end{equation}
    \end{enumerate}
\end{Lemma}
\begin{proof}
    Note that for $n\geq 1$ we have the following bound obtained by differentiating the integrand in the definition \eqref{def-nu},
    \begin{equation}\label{der_z}
    \partial_z^n\nu_{B, \varsigma}(x, y, z) \ll_n \frac{(\log B)^2}{B^{(\delta + \eta)n}}.
    \end{equation}
    Using the product rule and the above bound we deduce that
    \begin{equation}\label{zder}
   \partial_z^n \frac{\nu_{B, \varsigma}(x, y, z)}{|2z-xy|} \ll_n \frac{(\log B)^2}{B^{\delta n}}.
   \end{equation}
   As $M_a=0$, we have that $\partial_a z= 1/(2z-xy)$. Using the chain rule and the bound \eqref{zder} we obtain the bound
   \[
   \partial_a^n \frac{\nu_{B, \varsigma}(x, y, z)}{|2z-xy|} \ll_n \frac{(\log B)^2}{B^{(2-3\eta)n}},
   \]
   for $n\geq 2$. Now, differentiating the integrand in the definition \ref{real-den} $n> 2$ times and integrating we obtain the desired upper bound,
   \begin{equation*}
       \begin{split}
           \partial_a^n \sigma_{\infty,a,\nu}(B,\varsigma)&= \int_{(x,y)\in \RR^2, M_a=0} \partial_a^n\frac{\nu_{B,\varsigma}(x,y,z)}{\abs{2z-xy}}\, dx\,dy\\
           &\ll_n \int_{(x,y)\in \RR^2, M_a=0} \frac{(\log B)^2}{B^{(2-\delta-3\eta)n}B^{2-3\eta}}\, dx\,dy\\
           &\ll_n \frac{(\log B)^2}{B^{(2-\delta-3\eta)(n-2)}},
       \end{split}
   \end{equation*}
   we use the fact that $\nu_{B,\varsigma}$ is zero outside the region $|xy|\ll B^{2-\delta}$ for the last bound. This proves $(1)$. For $(2)$, we again note that for $i=1,2,3$ we have the bound
   \begin{equation}\label{bold_z}
       \partial_{z_i}^n\nu_{B, \varsigma}(\textbf{z}) \ll_n \frac{(\log B)^2}{B^{(\delta + \eta)n}}.
   \end{equation}
   Next, note that $2z\partial_{z_i}z-xy\partial_{z_i}z= 2z_i-\prod_{j\neq i}z_j$, since $M(x,y,z) = M(\textbf{z})$.
    Hence, $|\partial_{z_i}z|\ll B^{3\eta}$. Differentiating again we obtain,
    \[
    2(\partial_{z_i}z)^2 + (2z-xy)\partial_{z_i}^2z=2,
    \]
    therefore $|\partial_{z_i}^2z|\ll B^{6\eta}/B^{2-\delta-3\eta}$. Differentiating $k\geq 3$ times we obtain
    \[
    P_{\leq k} + (2z-xy)\partial_{z_i}^k z=0,
    \]
    where $P_{\leq k}$ is a linear combination of the monomials $(\partial_{z_i}^u z)(\partial_{z_i}^v z)$, where $u,v\ge 1$ and $u+v=k$.
    Using induction we deduce that
    \begin{equation}\label{partial_z_i}
        \partial_{z_i}^kz \ll_k \frac{B^{3\eta k}}{B^{(2-\delta-3\eta)(k-1)}}.
    \end{equation}
    Again using the chain rule along with the above bound we deduce that for $n > 2$,
    \[
    \partial_{z_i}^n\frac{\nu_{B, \varsigma}(x, y, z)}{|2z-xy|} \ll_n \frac{(\log B)^2}{B^{(\delta-3\eta)n}}.
    \]
    Hence, using the definition \ref{real-den} and differentiating the integrand we prove that
    \[
    \partial_n \sigma_{\infty,M(\textbf{z}),\nu}(B,\varsigma) \ll_n \frac{B^2(\log B)^2}{B^{(\delta-3\eta)n}}.
    \]
    Applying the product rule using the bound \eqref{bold_z} and the above bound we obtain that
    \[
    \partial_n \partial_{z_i}^n\nu_{B,\varsigma}(\textbf{z})\sigma_{\infty,M(\textbf{z}),\nu}(B,\varsigma) \ll_n \frac{B^2(\log B)^4}{B^{(\delta-\eta)n}}.
    \]
    For $(3)$, our strategy is similar to $(2)$ with $M$ replaced by $G_{a_1, a_2}$. We prove the bound for $a_2$ and note that the same bound applies to the partial derivatives with respect to $a_1$ due to symmetry between $a_1, a_2$ in the definition of $\sigma_{\infty, \nu}^{\otimes 2}(a_1, a_2)(B, \varsigma)$. Repeating the same strategy, we note that $2v_1\partial_{a_2}v_1 - a_1v_2\partial_{a_2}v_1= 2a_2-w_1w_2$ as $G_{a_1, a_2}=0$. Proceeding as in the proof of the bound \eqref{partial_z_i}, using induction we deduce that
    \begin{equation}\label{der_a_2}
       \partial_{a_2}^kv_1\ll \frac{B}{B^{(\delta+\eta)k}}
    \end{equation}
    Further, we also note the bound (similar to \eqref{zder})
    \[
    \partial_{v_1}^n \frac{\nu_{B, \varsigma}(v_1, v_2, a_1)}{|2v_1-v_2a_1|} \ll_n \frac{(\log B)^2}{B^{(1-3\eta) n}}.
    \]
    Using the chain rule with the the bound $\eqref{der_a_2}$, and the above bound, we deduce that
    \[
    \partial_{a_2}^n \frac{\nu_{B, \varsigma}(v_1, v_2, a_1)}{|2v_1-v_2a_1|} \ll_n \frac{(\log B)^2}{B^{(\delta-2\eta) n}},
    \]
   for $n> 2$. Next, using the bound \eqref{der_z}, the above bound, and applying the product rule we conclude that the following bound holds
   \[
   \partial_{a_2}\nu_{B, \varsigma}(w_1, w_2, a_2)\frac{\nu_{B, \varsigma}(v_1, v_2, a_1)}{|2v_1-v_2a_1|} \ll \frac{(\log B)^4}{B^{(\delta-\eta) n}}.
   \]
   Now, as in proof of $(1)$, differentiating the integrand in the definition of $\sigma_{\infty, \nu}^{\otimes 2}(a_1, a_2)(B, \varsigma)$ we complete the proof of the lemma.
\end{proof}

We conclude this section by deriving asymptotics for sums that involve the various real densities we have defined. The proof utilizes Lemmas \ref{int-eq} and \ref{der_bound}, as well as Poisson summation.
\begin{proposition}\label{real_sum}
    Let $B, N\geq 1$. Fix $b, d \in \Z/N\Z$ an $\textbf{e}\in (\Z/N\Z)^3$ then we have that
    \begin{equation}\label{real_sum_1}
        \sum_{a\in \Z: a\equiv b\bmod N}\sigma_{\infty,a,\nu}(B,\varsigma)^2= \frac{\sigma^{\otimes 2}_{\infty,\nu}(B,\varsigma)}{N} + O_n\left(\frac{N^{n-1}B^2(\log B)^4}{B^{(2-\delta-3\eta)(n-4)}}\right),
    \end{equation}
    \begin{equation}\label{real_sum_2}
        \sum_{\textbf{z}\in \Z^3: \textbf{z}\equiv \textbf{e}\bmod N}\nu_{B,\varsigma}(\textbf{z})\sigma_{\infty,M(\textbf{z}),\nu}(B,\varsigma)= \frac{\sigma^{\otimes 2}_{\infty,\nu}(B,\varsigma)}{N^3} + O_n(N^{n-3} B^{4-(\delta-\eta) n}(\log B)^4),
    \end{equation}
    and
    \begin{equation}\label{real_sum_3}
        \sum_{\substack{(a_1, a_2) \in \Z^2:\\ (a_1, a_2) \equiv (b, d)\bmod N}} \sigma^{\otimes 2}_{\infty,\nu}(a_1, a_2, B,\varsigma)= \frac{\sigma^{\otimes 2}_{\infty,\nu}(B,\varsigma)}{N^2} + O_n(N^{n-2}B^{4-(\delta-\eta) n}(\log B)^4).
    \end{equation}
\end{proposition}
\begin{proof}
    For \eqref{real_sum_1}, using Poisson summation and Lemma \ref{int-eq}, we deduce that
    \[
    \int_{a\in \RR}\sigma_{\infty,a,\nu}(B,\varsigma)\,da= \frac{\sigma^{\otimes 2}_{\infty,\nu}(B,\varsigma)}{N} + \sum_{m\neq 0}\frac{O(1)}{N}\left|\int_{a\in \RR}\sigma_{\infty,a,\nu}(B,\varsigma)^2e(-m\cdot a/N)\,da\right|.
    \]
     Note that using the bound \eqref{der1} from Lemma \ref{der_bound} we have the following bound,
    \[
    \partial_a^n \sigma_{\infty,a,\nu}(B,\varsigma)^2\ll_n \frac{(\log B)^4}{B^{(2- \delta-3\eta)(n-4)}},
    \]
    for $n > 4$.
    Integrating by parts $n> 4$ times we obtain that
    \begin{equation*}
        \begin{split}
            \int_{a\in \RR}\sigma_{\infty,a,\nu}(B,\varsigma)^2e(-m\cdot a/N)\,da&\ll_n \frac{N^nB^2(\log B)^4}{m^nB^{(2-\delta-3\eta)(n-4)}}
        \end{split}
    \end{equation*}
To prove \eqref{real_sum_2} we use Poisson summation again to obtain,
    \begin{equation*}
        \begin{split}
            \sum_{\textbf{z}\in \Z^3: \textbf{z}\equiv \textbf{e}\bmod N}\nu_{B,\varsigma}&(\textbf{z})\sigma_{\infty,M(\textbf{z}),\nu}(B,\varsigma)= \frac{\sigma^{\otimes 2}_{\infty,\nu}(B,\varsigma)}{N^3}\\&+ \sum_{\textbf{c}\in \Z^3- 0}\frac{O(1)}{N^3}\left|\int_{\textbf{z}\in \RR^3}\nu_{B,\varsigma}(\textbf{z})\sigma_{\infty,M(\textbf{z}),\nu}(B,\varsigma)e(-\textbf{c}\cdot \textbf{z}/N)\,d\textbf{z}\right|.
        \end{split}
    \end{equation*}
    Let $\textbf{c}=(c_1, c_2, c_3) \neq \textbf{0}$, assume $|c_j|=\|\textbf{c}\|_{\max}$. Integrating by parts $n \geq 4$ times and using the bound \eqref{der2} from Lemma \ref{der_bound}
   \[
    \partial_{z_j}^n\nu_{B,\varsigma}(\textbf{z})\sigma_{\infty,M(\textbf{z}),\nu}(B,\varsigma) \ll_n \frac{B^2(\log B)^4}{B^{(\delta-\eta)n}},
    \]
    we prove that
    \[
    \int_{\textbf{z}\in \RR^3}\nu_{B,\varsigma}(\textbf{z})\sigma_{\infty,M(\textbf{z}),\nu}(B,\varsigma)e(-\textbf{c}\cdot \textbf{z}/N)\,d\textbf{z}\ll \frac{N^nB^{4-(\delta-\eta) n}(\log B)^4}{\|c\|_{\max}^n}.
    \]
    Next, we prove \eqref{real_sum_3} using the same method. The error term is
    \[
    \sum_{\textbf{c}\in \Z^2- 0}\frac{O(1)}{N^2}\left|\int_{\textbf{a}\in \RR^2}\sigma^{\otimes 2}_{\infty,\nu}(a_1, a_2, B,\varsigma)e(-\textbf{c}\cdot \textbf{a}/N)\,d\textbf{a}\right|.
    \]
    Again let us assume that $|c_j|=\|\textbf{c}\|_{\max} \neq 0$ where $\textbf{c}= (c_1, c_2)$. Integrating by parts $n \geq 3$ times and using the bound \eqref{der3} from Lemma \ref{der_bound} we conclude that
    \[
    \int_{\textbf{a}\in \RR^2}\sigma^{\otimes 2}_{\infty,\nu}(a_1, a_2, B,\varsigma)e(-\textbf{c}\cdot \textbf{a}/N)\,d\textbf{a} \ll \frac{N^nB^{4-(\delta-\eta) n}(\log B)^4}{\|c\|_{\max}^n}.
    \]
    This concludes the proof of the Proposition.
\end{proof}

\section{Variance analysis}\label{variance-analysis}
In this section, we combine the results from Section \ref{preliminaries} to derive an upper bound for the approximate variance $\op{Var}(B,\varsigma, K)$ as defined in \eqref{var-def}, when $K$ is a small power of $B$. This upper bound, together with a small value estimate, leads to the proof of Theorem \ref{main-thm} in Section \ref{last section}. Let us first present the upper bound.
\begin{theorem}\label{variance_bound}
    Let $B, K\geq 3$ be integers. There exists some $\delta_0\in (0,1)$ such that for $K\leq B^{\delta_0}$ we have the bound
    \[
    \op{Var}(B,\varsigma,K)\ll B^2(\log B)^2.
    \]
\end{theorem}
To establish our desired upper bound, we first express the variance $\op{Var}(B,\varsigma, K)$, as a sum of three terms. This approach is similar to the methods outlined in \cite{ghosh2017integral} and \cite{wang2024prime}. By obtaining asymptotic bounds for each of these terms with appropriate cancellation, we prove the upper bound in Theorem \ref{variance_bound}. We begin by expanding the variance
\[
\op{Var}(B,\varsigma,K)\defeq \sum_{a\in \Z} (r_a(B,\varsigma) - s_a(K)\sigma_{\infty,a,\nu}(B,\varsigma))^2:= \Sigma_1-2\Sigma_2+\Sigma_3,
\]
where
\[
\Sigma_1 \defeq \sum_{a\in \Z} r_a(B,\varsigma)^2,\,\Sigma_2 \defeq \sum_{a\in \Z} r_a(B,\varsigma) s_a(K)\sigma_{\infty,a,\nu}(B,\varsigma),
\]
and
\[
\Sigma_3 \defeq \sum_{a\in \Z} (s_a(K)\sigma_{\infty,a,\nu}(B,\varsigma))^2.
\]
We analyze each sum separately.
\subsection{Estimating the three sums}
First, we consider the sum $\Sigma_3$. Using the definition for $s_a(-)$ from \eqref{definition_s_a} and grouping terms modulo $n_1, n_2$ we rewrite it as
\begin{equation*}
    \begin{split}
        \Sigma_3 &= \sum_{a\in \Z} (s_a(K)\sigma_{\infty,a,\nu}(B,\varsigma))^2\\
        &=\sum_{n_1, n_2 \leq K}\sum_{b \in \Z/n_1n_2\Z} (n_1n_2)^{-3}T_b(n_1)T_b(n_2)\sum_{a\in \Z: a\equiv b\mod n_1n_2}\sigma_{\infty,a,\nu}(B,\varsigma)^2.
    \end{split}
\end{equation*}
From \eqref{real_sum_1} in Proposition \ref{real_sum} with $n_0= 5$ and using the trivial bound $|T_b(n)|\leq n^4$ we conclude that
\begin{equation*}
    \Sigma_3= \sum_{n_1, n_2 \leq K}\sum_{b \in \Z/n_1n_2\Z} (n_1n_2)^{-3}T_b(n_1)T_b(n_2)\frac{\sigma^{\otimes 2}_{\infty,\nu}(B,\varsigma)}{n_1n_2}+O\left(K^{10}(\log B)^2\right)
\end{equation*}
Proposition \ref{imp-sum} simplifies the main term and we obtain
\begin{equation}\label{sigma_3}
    \Sigma_3= \sigma^{\otimes 2}_{\infty,\nu}(B,\varsigma)\sum_{m\leq K} \frac{S_{\textbf{0}}(m)}{m^6} +O\left(K^{10}(\log B)^2\right).
\end{equation}
\par Next, we consider the sum $\Sigma_2$. Expanding this sum using the definition of $r_a(B,\varsigma)$ we obtain that
\begin{equation*}
    \begin{split}
        \Sigma_2&= \sum_{a\in \Z} r_a(B,\varsigma) s_a(K)\sigma_{\infty,a,\nu}(B,\varsigma),\\
         &= \sum_{n\leq K}\sum_{\textbf{e}\in (\Z/n\Z)^3}n^{-3}T_{M(\textbf{e})}(n)\sum_{\textbf{z}\in \Z^3: \textbf{z}\equiv \textbf{e}\mod N}\nu_{B,\varsigma}(\textbf{z})\sigma_{\infty,M(\textbf{z}),\nu}(B,\varsigma).
    \end{split}
\end{equation*}
Using \eqref{real_sum_2} from Proposition \ref{real_sum} with $n_1$ suitably large for the error term and using the trivial bound $|T_b(n)|\leq n^4$ we conclude that
\begin{equation*}
    \Sigma_2= \sum_{n\leq K}\sum_{\textbf{e}\in (\Z/n\Z)^3}n^{-3}T_{M(\textbf{e})}(n)\frac{\sigma^{\otimes 2}_{\infty,\nu}(B,\varsigma)}{n^3} + O(K^{n_1}(\log B)^2).
\end{equation*}
Applying Proposition \ref{imp-sum} in this case as well we deduce that
\begin{equation}\label{sigma_2}
    \Sigma_2= \sigma^{\otimes 2}_{\infty,\nu}(B,\varsigma)\sum_{m\leq K} \frac{S_{\textbf{0}}(m)}{m^6} + O(K^{n_1}(\log B)^2).
\end{equation}
\par Next, we analyze the sum $\Sigma_1$. For the diagonal terms, we require the following lemma, and we address the non-diagonal terms using results from Section \ref{preliminaries}.
\begin{Lemma}\label{LEM:z-diagonal-estimate}
Given $Y_1,Z_1,Y_2,Z_2$, and $z_1=z_2=z$, we have
\begin{equation*}
\sum_{\substack{x_1,y_1,x_2,y_2\in \Z \\ M(x_1,y_1,z) = M(x_2,y_2,z)}}
\prod_{1\le i\le 2}
\nu{\left(\frac{y_i}{Y_i}\right)}
\nu{\left(\frac{x_i}{X_i}\right)}
- \sum_{x,y\in \Z}
\prod_{1\le i\le 2}
\nu{\left(\frac{y}{Y_i}\right)}
\nu{\left(\frac{x}{X_i}\right)}
\ll_\eps B^\eps Y_1Y_2.
\end{equation*}
Moreover, $B^{\delta-4\eta} Y_1Y_2\le \min(X_1Y_1,X_2Y_2)$.
\end{Lemma}
\begin{proof}
If $B\ll 1$, then the left-hand side is trivially $O(1) \le Y_1Y_2$.
So we may assume $B\gg 1$.

By completing the square as in \cite[\S9]{ghosh2017integral}, we may write
\begin{equation*}
4(M(x,y,z)-z^2) = 4(x^2+y^2-zxy) = (2x - zy)^2 + (4-z^2)y^2.
\end{equation*}
If $y_1=y_2=y$, then $M(x_1,y_1,z) = M(x_2,y_2,z)$ implies that
either $x_1=x_2$, or $x_1+x_2 = zy$.
The latter case is impossible, because $\abs{x_1+x_2} \le 2B^{1-\eta} < B^{1+\eta} < \abs{zy}$ for $B\gg_\eta 1$.

Therefore, the left-hand side of the desired inequality equals
\begin{equation*}
\sum_{\substack{x_1,y_1,x_2,y_2\in \Z \\ M(x_1,y_1,z) = M(x_2,y_2,z)}}
\textbf{1}_{y_1\ne y_2}
\prod_{1\le i\le 2}
\nu{\left(\frac{y_i}{Y_i}\right)}
\nu{\left(\frac{x_i}{X_i}\right)}.
\end{equation*}
But for any given $y_1\asymp Y_1$ and $y_2\asymp Y_2$ with $y_1\ne y_2$,
there are at most $O_\eps(z^\eps \max(Y_1,Y_2)^\eps)$ choices for $x_1,x_2\in \Z$ satisfying $M(x_1,y_1,z) = M(x_2,y_2,z)$,
by the divisor bound on the nonzero integer $(4-z^2)(y_1^2-y_2^2)$.
On summing over $y_1$ and $y_2$, we conclude that the last display is $\ll_\eps B^\eps Y_1Y_2$.
\end{proof}
We rewrite the sum $\Sigma_1$ using the definition \eqref{definition-r_a} and further express it as a sum of terms as follows
\begin{equation*}
    \begin{split}
        \Sigma_1&= \sum_{\tM=0}\nu_{B,\varsigma}(x_1, y_1, z_1)\nu_{B,\varsigma}(x_2, y_2, z_2),\\
        &= \Sigma_{1,d}+\Sigma_{1,nd},
    \end{split}
\end{equation*}
where
\[
\Sigma_{1,d}:= \sum_{\tM=0; z_1=z_2}\nu_{B,\varsigma}(x_1, y_1, z_1)\nu_{B,\varsigma}(x_2, y_2, z_2),
\]
and
\[
\Sigma_{1,nd}:= \sum_{\tM=0; z_1\neq z_2}\nu_{B,\varsigma}(x_1, y_1, z_1)\nu_{B,\varsigma}(x_2, y_2, z_2).
\]
First, we deal with the sum $\Sigma_{1,d}$. Interchanging the order of summation and integral we obtain that
\begin{equation*}
\begin{split}
\Sigma_{1,d}-\sum_{x,y,z}&\nu_{B,\varsigma}(x, y, z)^2\\
&= \sum_z\int_{B^{\delta+\eta}}^{B^{\delta+2\eta}}\int_{B^{1-\delta}}^{B^{1-\delta+\eta}}\int_{B^{\delta+\eta}}^{B^{\delta+2\eta}}\int_{B^{1-\delta}}^{B^{1-\delta+\eta}}\Delta\cdot \nu{\left(\frac{z}{Z_1}\right)}\nu{\left(\frac{z}{Z_2}\right)} \prod_i \frac{dY_i\,dZ_i}{Y_iZ_i},
\end{split}
\end{equation*}
where
\[
\Delta:= \sum_{\substack{x_1,y_1,x_2,y_2\in \Z \\ M(x_1,y_1,z) = M(x_2,y_2,z)}}
\prod_{1\le i\le 2}
\nu{\left(\frac{y_i}{Y_i}\right)}
\nu{\left(\frac{x_i}{X_i}\right)}
- \sum_{x,y\in \Z}
\prod_{1\le i\le 2}
\nu{\left(\frac{y}{Y_i}\right)}
\nu{\left(\frac{x}{X_i}\right)}
\]
is the sum from the Lemma \ref{LEM:z-diagonal-estimate}. Leveraging the bound from Lemma \ref{LEM:z-diagonal-estimate} and integrating, we deduce that
\[
\Sigma_{1,d}-\sum_{x,y,z}\nu_{B,\varsigma}(x, y, z)^2 \ll B^2(\log B)^2,
\]
and we have the trivial bound
\[
\sum_{x,y,z}\nu_{B,\varsigma}(x, y, z)^2 \ll B^2(\log B)^2.
\]
Hence, we deduce that
\begin{equation}\label{bound_s_d}
   \Sigma_{1,d} \ll B^2(\log B)^2. 
\end{equation}

Next, we analyze $\Sigma_{1,nd}$ using the Kloosterman circle method with uniformity in the error term as discussed in the Appendix.
Given $Y_1, Z_1, Y_2$, and $Z_2$, and using Theorem~\ref{blackbox} along with the conditions in \eqref{basic-support-conditions}, we deduce for any $a_1 \neq a_2$, the estimate
\begin{equation*}
\sum_{\substack{x_1, y_1, x_2, y_2 \in \Z \\ M(x_1, y_1, a_1) = M(x_2, y_2, a_2)}} \prod_{1 \le i \le 2} \nu{\left(\frac{y_i}{Y_i}\right)} \nu{\left(\frac{x_i}{X_i}\right)} = \sigma_\infty(a_1, a_2) \mathfrak{S} + O_\epsilon(B^{3/2 + 1729\delta + \epsilon}).
\end{equation*}
Here we used Theorem~\ref{blackbox} with the weight function $\nu$ and our quadratic form $F_{a_1, a_2} = x_1^2 + y_1^2 - a_1 x_1 y_1 - x_2^2 - y_2^2 + a_2 x_2 y_2$. We chose the parameters $P = B$, $\lambda_{2i-1} = B / Y_i$, and $\lambda_{2i} = B / X_i$ for $i = 1, 2$.  Multiplying by $\prod_{i=1,2} v(a_i / Z_i) / Y_i Z_i$ and integrating as in the definition \eqref{def-nu}, we deduce the following estimate:
\begin{equation*}
\begin{split}
\sum_{\tM=0} \nu_{B,\varsigma}(x_1, y_1, a_1) \nu_{B,\varsigma}(x_2, y_2, a_2) = \sigma^{\otimes 2}_{\infty,\nu}(a_1, a_2, B, \varsigma) & \sum_{m \geq 1} m^{-4} S_{\mathbf{0}}^+(a_1, a_2; m) \\
& + O_\epsilon(B^{3/2 + 1729\delta + \epsilon} (\log B)^4).
\end{split}
\end{equation*}
Truncating the singular series at $K$ and summing over all possible values of integers $a_1, a_2$ we obtain that
\begin{equation}\label{circle}
    \begin{split}
        \Sigma_{1,nd}&= \sum_{a_1\neq a_2}\sigma^{\otimes 2}_{\infty,\nu}(a_1, a_2, B,\varsigma)\sum_{m\leq K} m^{-4}S_{\textbf{0}}(a_1, a_2; m)\\
        & + \sum_{a_1\neq a_2}\sigma^{\otimes 2}_{\infty,\nu}(a_1, a_2, B,\varsigma)\sum_{m> K} m^{-4}S_{\textbf{0}}(a_1, a_2; m) + O(B^{2}(\log B)^2).
    \end{split}
\end{equation}
For the error term, we assume that $\delta \ll 1$. We bound both the terms separately. First, we bound the term involving the tail of the singular series. We need the following lemma.
\begin{Lemma}\label{gcd_bound}
    Let $A_1, A_2, D\geq 1$ be integers suppose that $D \ll A_2$, then we have the following bound
    \[
    \sum_{\substack{a_1\in [A_1, 2A_1],\\ a_2\in [A_2, 2A_2]}} \mathds{1}_{\gcd(a_1^2-4, a_2^2-4)\geq D} \ll_\eps \frac{A_1^{3\eps}A_1A_2}{D}.
    \]
\end{Lemma}
\begin{proof}
    Fix $a_1\in [A_1, 2A_1]$. Let $d\geq D$ be an integer such that $d\mid a_1^2-4$, the number of possibilities for $a_2\in [A_2, 2A_2]$ such that $d\mid a_2^2-4$ is $O_\eps(A_2d^\eps/d+d^\eps)$. We note that
    \[
    \frac{A_2d^\eps}{d}+d^\eps \ll \frac{A_1^{2\eps}A_2}{D}.
    \]
    Noting that there are $O_\eps(A_1^{\eps})$ possibilities for $d\geq D$ for any $a_1\in[A_1, 2A_1]$ and summing over $a_1\in [A_1, 2A_1]$, we deduce that
    \[
    \sum_{\substack{a_1\in [A_1, 2A_1],\\ a_2\in [A_2, 2A_2]}} \mathds{1}_{\gcd(a_1^2-4, a_2^2-4)\geq D} \ll_\eps \frac{A_1^{3\eps}A_1A_2}{D}.
    \]
\end{proof}
We now use the upper-bound for the truncated singular series from Lemma \ref{uniform_singular} and the upper-bound \eqref{bound_real_a_1_2} on the real density from the Lemma \ref{bound_real_a_1_2_lemma}  to conclude that
\begin{equation}\label{Sigma_1_1_h}
\begin{split}
    \sum_{a_1\neq a_2}\sigma^{\otimes 2}_{\infty,\nu}&(a_1, a_2, B,\varsigma)\sum_{m> K} m^{-4}S_{\textbf{0}}(a_1, a_2; m)\\
    &\ll_\eps B^{2+\eps}(\log B)^2\sum_{a_1\neq a_2}\frac{1}{|a_1||a_2|}\min\left(1, \frac{\gcd(a_1^2-4, a_2^2-4)^5}{K}\right).
\end{split}
\end{equation}
We consider the sum
\begin{equation}\label{gcd_bound_2}
    \begin{split}
        \sum_{\substack{a_1\neq a_2,\\\gcd(a_1^2-4, a_2^2-4)\geq K^{1/10}}}\frac{1}{|a_1||a_2|} \ll_\eps \frac{B^\eps(\log B)^2}{K^{1/10}},
    \end{split}
\end{equation}
we obtain this bound by dividing the sum above into dyadic intervals and applying Lemma \ref{gcd_bound} with $D=K^{1/10}$ assuming $K \ll B^{\delta}$. As a consequence of \eqref{Sigma_1_1_h} and the bound \eqref{gcd_bound_2} we have that
\begin{equation}\label{Sigma_1_1}
    \sum_{a_1\neq a_2}\sigma^{\otimes 2}_{\infty,\nu}(a_1, a_2, B,\varsigma)\sum_{m> K} m^{-4}S_{\textbf{0}}(a_1, a_2; m) \ll_\eps \frac{B^{2+\eps}(\log B)^4}{K^{1/10}}.
\end{equation}
Now we bound the sum
\begin{equation*}
    \sum_{a_1\neq a_2}\sigma^{\otimes 2}_{\infty,\nu}(a_1, a_2, B,\varsigma)\sum_{m\leq K} m^{-4}S_{\textbf{0}}(a_1, a_2; m).
\end{equation*}
By interchanging the order of summation and grouping terms modulo 
$m$, we deduce that the above sum is equivalent to
\begin{equation*}
    \sum_{m\leq K}\sum_{b_1, b_2\in \Z/m\Z}m^{-4}S_{\textbf{0}}(b_1, b_2; m)\sum_{\substack{a_1, a_2\equiv b_1, b_2 \mod m,\\a_1\neq a_2}}\sigma^{\otimes 2}_{\infty,\nu}(a_1, a_2, B,\varsigma).
\end{equation*}
Using \eqref{real_sum_3} from Proposition \ref{real_sum} with suitable integer $n_2$ and the bound \eqref{bound_real_a_1_2} for diagonal terms with $a_1=a_2$ we deduce that the above sum is
\[
\sum_{m\leq K}\sum_{b_1, b_2\in \Z/m\Z}m^{-6}S_{\textbf{0}}(b_1, b_2; m)\sigma^{\otimes 2}_{\infty,\nu}(B,\varsigma) + O(K^{n_2}B^{2-\delta}(\log B)^4).
\]
Therefore, we deduce that
\begin{equation}\label{Sigma_1_2}
\begin{split}
     \sum_{a_1\neq a_2}\sigma^{\otimes 2}_{\infty,\nu}(a_1, a_2, B,\varsigma)&\sum_{m\leq K} m^{-4}S_{\textbf{0}}(a_1, a_2; m)\\
     &= \sigma^{\otimes 2}_{\infty,\nu}(B,\varsigma)\sum_{m\leq K} \frac{S_{\textbf{0}}(m)}{m^6} + O(K^{n_2}B^{2-\delta}(\log B)^4).
\end{split}
\end{equation}
\subsection{Proof of Theorem \ref{variance_bound}}
The proof is now straightforward. Let $N_0\geq n_0, n_1, n_2$ $\delta_0= \delta/(N_0\times1000)$. By combining the estimates of the three sums $\Sigma_1, \Sigma_2,$ and $\Sigma_3$ from \eqref{sigma_3}, \eqref{sigma_2}, \eqref{bound_s_d}, \eqref{circle}, \eqref{Sigma_1_1}, \eqref{Sigma_1_2} we obtain the bound
\[
\op{Var}(B,\varsigma,K)\ll B^2(\log B)^2,
\]
for $B, K \geq 3$ integers with $K\leq B^{\delta_0}$.
\section{A small value estimate}\label{small_est}
In this section we establish a small value estimate analogous to \cite[Theorem 3.9]{wang2024prime}. Our strategy is similar to Wang's method, and is based on local computations from \cite{ghosh2017integral}. To prove bounds $s_a(K)$ we analyze the approximate inverse, given by
\[
t_a(K):= \sum_{\substack{n \leq K\\ \gcd(n, 30)=1}}\mu(n)n^{-3}T_a(n).
\]
Following \cite{wang2024prime} we define $T_a^\natural(n):= n^{-2}T_a(n)$ to simplify notation. The next few results establish bounds on $T_a^\natural$ and they follow directly from \cite[Appendix B]{ghosh2017integral}.

\begin{Lemma}\label{bound_T_1}
    Let $a \in \Z$. We have the following bounds
    \[
    T_a^\natural(2)\leq \frac{1}{4},
    \]
    and
    \[
    T_a^{\natural}(p)\leq 4+\frac{1}{p},
    \]
    for all odd primes $p$.
\end{Lemma}
\begin{proof}
    This follows directly from \cite[Lemma B.9]{ghosh2017integral} and \cite[Proposition B.1]{ghosh2017integral}.
\end{proof}
Next lemma deals with values of $T_a^{\natural}$ at higher prime powers.
\begin{Lemma}\label{bound_T_2}
    Let $a \in \Z$ and $\ell \in \Z_{\geq 2}$, then we have
    \[
    T_a^\natural(p^\ell) \ll p^{\ell/2}.
    \]
\end{Lemma}
\begin{proof}
    Immediate from \cite[Lemma B.9, B.10]{ghosh2017integral} and \cite[Proposotion B.1]{ghosh2017integral}.
\end{proof}
We also need certain vanishing results for $T_a^\natural$ that will be useful in deriving the desired small value estimate. The following lemma is again a consequence of \cite[Proposition B.1]{ghosh2017integral}.
\begin{Lemma}\label{vanish_T_p_odd}
    Let $p$ be an odd prime and let $\ell\geq 2$. Suppose $p^{\ell-1} \nmid a(a-4)$ then $T_a^\natural(p^\ell)=0$.
\end{Lemma}
We write the product $s_a(K)t_a(K)$ as two sums as in \cite[pp. 14]{wang2024prime}.
\begin{equation}\label{def_inv_s}
    s_a(K)t_a(K)= \sum_{n\leq K}c_a(n)+ \sum_{\substack{n_1,n_2;\\ n_1n_2 > K, \gcd(n_2, 30)=1}}n_1^{-3}T_a(n_1)\cdot\mu(n_2)n_2^{-3}T_a(n_2),
\end{equation}
where
\begin{equation}\label{c_a-def}
    c_a(n)\defeq n^{-3}\sum_{d\mid n, \gcd(n_2, 30)=1} T_a(d)\cdot\mu(n/d)T_a(n/d).
\end{equation}
Using \cite[Lemma 2.13]{browningcubic} we deduce that $T_a$ is multiplicative. Hence, $c_a$ is multiplicative. For any prime $p$ and $\ell \geq 1$, using \eqref{c_a-def} we find that
\[
c_a(p^\ell)= p^{-\ell}[T_a^\natural(p^\ell)-T_a^{\natural}(p^{\ell-1})T_a^\natural(p)\mathds{1}_{p\geq 7}].
\]
Our next result is analog of \cite[Lemma 3.4]{wang2024prime} establishing an upper bound on $c_a$. We bound $c_a$ at prime powers using the lemmas stated earlier and then use the multiplicativity of $c_a$ to conclude the proof.
\begin{proposition}\label{bound_c_a}
    Let $a \in \Z-\{0\}$ and $n \in \Z_{\geq 1}$. Then
    \[
    c_a(n) \leq O(1)^{\omega(n)}n_1^{-2}n_2^{-1}\gcd(n_2^{1/2}, a(a-4))n_3^{-1/2}.
    \]
    where $n_3:= \op{cub}(n)$, $n_2:=\op{sq}(n/n_3)$, and $n_1:=n/n_2n_3$.
\end{proposition}
\begin{proof}
    Let $p$ be a prime. Note that if $\ell=1$ then $c_a(p^\ell)\ll 1/p^2$ trvially. For $\ell=2$ we conclude that $c_a(p^\ell)= p^{-2}[T_a^\natural(p^2)-T_a^\natural(p)^2] \ll 1/p^2$ if $p\nmid a(a-4)$ and $c_a(p^\ell) \ll 1/p$ using Lemmas \ref{bound_T_1} and \ref{bound_T_2}. Now suppose $\ell \geq 3$, then we deduce that $c_a(p^\ell)= p^{-\ell}[T_a^\natural(p^\ell)-T_a^{\natural}(p^{\ell-1})T_a^\natural(p)] \ll 1/p^{\ell/2}$ using the Lemma \ref{bound_T_2}. We obtain the result by putting these bounds together using the multiplicativity of $c_a$.
\end{proof}
We now prove the analog of \cite[Proposition 3.5]{wang2024prime} controlling the partial sums from $c_a$.
\begin{proposition}\label{sum_c_a}
    Let $a \in \Z-\{0\}$ such that $a \not\equiv 3 \bmod 4$ and $a \not\equiv \pm 3 \bmod 9$. Let $D, K \in \Z$ and $p$ a prime.
    \begin{enumerate}
        \item The sum $\sum_{\ell\geq 0}c_a(p^\ell)$ converges absolutely to a positive real number $\gamma_p(a)$.
        \item We have that $\gamma_p(a)= 1 + O(p^{-2})$ if $p\nmid a(a-4)$ and $\gamma_p(a) \geq (1-p^{-1})^{O(1)}$ if $p\mid a(a-4)$.
        \item The product $\prod_p \gamma_p(a)$ converges absolutely to a real number $\gamma(a) \gg \prod_{p\mid a(a-4)}(1-p^{-1})^{O(1)}$.
        \item We have that $\sum_{n\leq K}c_a(n)= \gamma(a) + O(|a(a-4)|^\epsilon K^{\eps-1/6})$.
    \end{enumerate}
\end{proposition}
\begin{proof}
    From the bound in the Proposition \ref{bound_c_a} we deduce that $\sum_{\ell\geq 0}c_a(p^\ell)$ converges absolutely. Further, we note that for $\ell\geq 1$
    \[
    c_a(p^\ell)= p^{-3\ell}T_a(p^\ell)-p^{-3(\ell-1)}T_a(p^{\ell-1})\cdot p^{-3}T_a(p)\mathds{1}_{p\geq 7}.
    \]
    Hence, we conclude that
    \begin{equation}
    \begin{split}
        \sum_{\ell\geq 0}c_a(p^\ell)&= (1-p^{-3}T_a(p)\mathds{1}_{p\geq 7})\left(1+\sum_{\ell\geq 1}p^{-3\ell}T_a(p^\ell)\right)\\
         &= (1-p^{-3}T_a(p)\mathds{1}_{p\geq 7})\delta_p(a).
    \end{split}
    \end{equation}
    Using Lemma \ref{bound_T_1} we get that $(1-p^{-3}T_a(p)\mathds{1}_{p\geq 7}) \geq 7/16$ and $\delta_p(a) > 0$ as $a \not\equiv 3 \bmod 4$ and $a \not\equiv \pm 3 \bmod 9$. This proves $(1)$.\newline
    Assertion $(2)$ follows from \cite[Proposition B.1, B.2]{ghosh2017integral} and $(3)$ is a direct consequence of $(2)$.\newline
    To prove $(4)$, we follow \cite{wang2024prime} very closely. By Proposition \ref{bound_c_a}
    \[
    \sum_{n > K}c_a(n)\ll_\eps \sum_{\substack{n_1, m_2, n_3\geq 1,\\ n_1m_2^2n_3> K,\\ n_3 \op{cubefull}}} (n_1m_2^2n_3)^\eps n_1^{-2}m_2^{-2}\gcd(m_2, a(a-4))n_3^{-1/2}.
    \]
    Let $N_1, M_2, N_3 \geq 1$ such that $(2N_1\cdot 2M_2^2\cdot N_3 > K)$. Bounding the contribution on the right-hand side coming from $(n_1, m_2, n_3) \in [N_1, 2N_1]\times [M_2, 2M_2]\times [N_3, 2N_3]$ as atmost $O((N_1M_2^2N_3)^\eps)$ times
    \[
    \frac{N_1N_3^{1/3}}{N_1^2M_2^1N_3^{1/2}}\sum_{d\mid a(a-4)}\sum_{m_2 \in [M_2, 2M_2]; d\mid m2}d \ll \frac{1}{(N_1M_2^2N_3)^{1/6}}\sum_{d\mid a(a-4)} 1 \ll \frac{|a(a-4)|^\eps K^{3\eps-1/6}}{(N_1M_2^2N_3)^{3\eps}}.
    \]
    Hence, we deduce that $\sum_{n > K}c_a(n)\ll_\eps |a(a-4)|^\eps K^{3\eps-1/6}$. Therefore, the sum $\sum_{n\geq 1}c_n(a)$ converges to $\gamma(a)$ and $(4)$ holds.
\end{proof}
Now, we investigate the second sum on the right-hand side of \eqref{def_inv_s}. To proceed we prove analogs of \cite[Lemma 3.7, 3.8]{wang2024prime}. Given $\textbf{z}:= (z_1,\dots, z_u) \in \Z_{\geq 1}^u$, we define
\[
T_a^\natural(\textbf{z}):= \prod_{i=1}^uT_a^\natural(z_i).
\]
\begin{Lemma}\label{exp_T_a}
    Let $r \in \Z_{\geq 1}$ and $m_1, \dots, m_r, n_1,\dots, n_r$ be $2r$ positive integers such that $m_1\cdots n_r$ is sqaure-full. Then
    \[
    \mathbb{E}_{b \in \Z/m_1\cdots m_rn_1\cdots n_r\Z}[|T_b^\natural(\textbf{m})\cdot T_b^\natural(\textbf{n})|] \ll_{r, \eps} \frac{\prod_{i=1}^r(m_in_i)^{1/2+\eps}}{\op{rad}(m_1\cdots m_rn_1\cdots n_r)},
    \]
    where $\textbf{m}=(m_1,\dots,m_r)$ and $\textbf{n}= (n_1,\dots, n_r)$.
\end{Lemma}
\begin{proof}
    From multiplicativity, it follows that we can assume $m_1\cdots m_rn_1\cdots n_r= p^{\ell}$ with $\ell\geq 2$ for some prime $p$. If $p=2$, then using the bound from Lemma \ref{bound_T_2}, we conclude that
    \begin{equation*}
        \begin{split}
            \mathbb{E}_{b \in \Z/m_1\cdots m_rn_1\cdots n_r\Z}[|T_b^\natural(\textbf{m})\cdot T_b^\natural(\textbf{n})|] &\ll_{r, \eps} \prod_{i=1}^r(m_in_i)^{1/2+\eps}\\
             &\ll \frac{\prod_{i=1}^r(m_in_i)^{1/2+\eps}}{2}
        \end{split}
    \end{equation*}
    For $p\neq 2$, using Lemma \ref{bound_T_2} and the vanishing Lemma \ref{vanish_T_p_odd} we deduce that
    \[
    \mathbb{E}_{b \in \Z/m_1\cdots m_rn_1\cdots n_r\Z}[|T_b^\natural(\textbf{m})\cdot T_b^\natural(\textbf{n})|] \ll_{r, \eps} \frac{\prod_{i=1}^r(m_in_i)^{1/2+\eps}}{p}.
    \]
    This completes the proof.
\end{proof}

\begin{Lemma}\label{alg_geo_lemma}
    Let $A, K_1, K_2, j \in \Z_{\geq 1}$. Let $\beta: \Z^2\to \RR$ be a function supported on $[K_1, 2K_1)\times [K_2, 2K_2)$. for $a \in \Z$, we consider
    \[
    P_a:= \sum_{\substack{m,n\geq 1,\\ n \op{squarefree}}} \beta(m, n)T_a^\natural(m)T_a^\natural(n).
    \]
    Let $A\geq (K_1K_2)^{3j}$, then we have the following bound,
    \[
    \sum_{a\in [-A, A]}|P_a|^{2j} \ll_{j, \eps} A\cdot(K_1K_2)^{j+\eps}(\max |\beta|)^{2j}.
    \]
\end{Lemma}
\begin{proof}
    Let $r=2j$. Using Poisson summation (by replacing the sum with a smoothed sum over $a\in \Z$) we deduce that
    \[
    \sum_{a\in [-A, A]}|P_a|^r \ll \sum_{\substack{m_1,\dots,m_r \geq 1,\\ n_1,\dots, n_r\geq 1, \op{squarefree},\\ b\in \Z/m_1\cdots m_rn_1\cdots n_r\Z }} \beta(\textbf{m}, \textbf{n})T_b^\natural(\textbf{m})T_b^\natural(\textbf{n})\cdot\frac{A(1+O_k(K^{-k})}{m_1\cdots m_rn_1\cdots n_r}.
    \]
    where $\textbf{m}=(m_1,\dots,m_r)$, $\textbf{n}= (n_1,\dots, n_r)$, $\beta(\textbf{m}, \textbf{n}):= \prod_{i=1}^r\beta(m_i, n_i)$, and $K:= K_1K_2$.
    The error term on the right-hand side contributes at most $\ll_k K^{2r}\cdot (\max |\beta|)^r\cdot A\cdot K^r\cdot K^{-k}$. We pick $k=3r$ to control the error term. Next, we consider the main term on the right-hand side. Note that $\sum_{a\in \Z/p\Z}T_a(p)=0$, hence we obtain that
    \[
    \sum_{b\in \Z/m_1\cdots m_rn_1\cdots n_r\Z}T_b^\natural(\textbf{m})T_b^\natural(\textbf{n})=0,
    \]
    if $m_1\cdots m_rn_1\cdots n_r$ is not squareful. When $m_1\cdots m_rn_1\cdots n_r$ is squareful then from Lemma \ref{exp_T_a} we have that
    \[
    \mathbb{E}_{b \in \Z/m_1\cdots m_rn_1\cdots n_r\Z}[|T_b^\natural(\textbf{m})\cdot T_b^\natural(\textbf{n})|] \ll_{r, \eps} \frac{(K)^{1/2+\eps}}{\op{rad}(m_1\cdots m_rn_1\cdots n_r)}.
    \]
    Now bounding as in the proof of \cite[Lemma 3.8]{wang2024prime},
    \begin{equation*}
        \begin{split}
            \sum_{\substack{m_1,\dots,m_r \geq 1,\\ \n_1,\dots, n_r\geq 1, \op{squarefree},\\ b\in \Z/m_1\cdots m_rn_1\cdots n_r\Z }} \beta(\textbf{m}, \textbf{n})T_b^\natural(\textbf{m})T_b^\natural(\textbf{n}) &\ll \sum_{R \leq (4K)^{r/2}}\sum_{\substack{m_1,\dots,m_r\mid R^{\infty}\\n_1,\dots,n_r\mid R^{\infty}}}\frac{(\max |\beta|)^rK^{1/2+\eps}}{R}\\
            &\ll_{r,\eps} (\max |\beta|)^rK^{1/2+\eps(1+r+r^2)},
        \end{split}
    \end{equation*}
    using the bound $O_{r, \epsilon}(K^{\eps(r+r^2)})$ for $\sum_{R \leq (4K)^{r/2}}\sum_{\substack{m_1,\dots,m_r\mid R^{\infty}\\n_1,\dots,n_r\mid R^{\infty}}}1/R$ from \cite[pp. 17]{wang2024prime}. This concludes the proof.
\end{proof}
We now prove an estimate that will be used to prove Theorem \ref{small_value_est}.
\begin{Lemma}\label{phi_bound} For $A \geq 1$,
\[
\sum_{a\in [-A, A]-\{0\}}\left(\frac{|a(a-4)|}{\phi(|a(a-4)|)}\right)^\kappa \ll_{\kappa} A.
\]
\end{Lemma}
\begin{proof}
    Recall that
    \[
    \left(\frac{|a(a-4)|}{\phi(|a(a-4)|)}\right)^\kappa \ll_{\kappa, \eps} \sum_{d\mid a(a-4)}\frac{1}{d^{1-\eps}}.
    \]
    Hence, using the right-hand side above, the sum in the Lemma is at most
    \[
    \sum_{k_1, k_2}\sum_{\substack{d_1, d_2\\ \op{odd}}}\sum_{\substack{|a|\leq A,\\ 2^{k_1}d_1\mid a,\\2^{k_2}d_2\mid a-4}} \frac{1}{(2^{k_1+k_2}d_1d_2)^{1-\eps}}\ll \sum_{k_1, k_2}\sum_{\substack{d_1, d_2\\ \op{odd}}} \frac{A}{(2^{k_1+k_2}d_1^2d_2^2)^{1-\eps}}.
    \]
    The right-hand side is at most $O_{\eps}(A)$. This concludes the proof.
\end{proof}
Next, we use previous Lemmas to prove the small value estimate mentioned at the beginning of this Section. We begin with a definition.
\begin{definition}
    Let $\eta > 0$ and $A, K \in \Z_{\geq 1}$, we define
    \[
    \mathcal{S}(A, K, \eta):= \{a\in \Z: |a|\leq A, a\not\equiv 3 \bmod 4\text{ and }a\not\equiv \pm 3\bmod 9, |s_a(K)|\leq \eta \}
    \]
\end{definition}
We state the theorem.
\begin{theorem}\label{small_value_est}
    Let $\eta, \eps > 0$ and $A, K, j \geq 1$. If $A \geq K^{6j}$ then,
    \[
    \frac{\#\mathcal{S}(A, K, \eta)}{A}\ll_{j, \eps} \eta^{1.8j}+ A^{\eps}K^{-1.5j} + \eta^{0.2j}K^{-j+\eps}.
    \]
\end{theorem}
\begin{proof}
    We consider the sets
    \begin{enumerate}
        \item $\mathcal{S}_1:=\{a\in \Z: |a|\leq A, |t_a(K)|\geq \eta^{-9/10}\}$,
        \item $\mathcal{S}_2:= \{a \in \Z: |a|\leq A, \prod_{p\mid a(a-4)}(1-p^{-1})^C\leq C\cdot(\eta^{1/10} + A^{2\eps}K^{\eps-1/6})\}$,
        \item $\mathcal{S}_3:=\{a \in \Z: |a|\leq A, |s_a(K)t_a(K)-\sum_{n\leq K}c_a(n)|\geq\eta^{1/10}\}$,
    \end{enumerate}
    where $C>0$ is a constant such that Proposition \ref{sum_c_a} holds. Then using Proposition \ref{sum_c_a} we deduce that $\mathcal{S}(A, K, \eta)\subseteq \mathcal{S}_1\cup\mathcal{S}_2\cup \mathcal{S}_3$. We bound the cardinality of these sets separately. For $\mathcal{S}_1$ using Lemma \ref{alg_geo_lemma} for $t_a(K)$ with $\beta(m,n):= \mu(n)/n$, we deduce that
    \[
    \#\mathcal{S}_1 \leq \sum_{a\in [-A, A]}|t_a(K)|^{2j}\eta^{18j/10}\ll_{j} A\eta^{18j/10}.
    \]
    as $A \geq K^{6j}$. For $\mathcal{S}_2$, we again use Lemma \ref{alg_geo_lemma} with $\beta(m,n):=\mu(n)/mn$ supported on $[K,2K)\times [K,2K)$ to obtain the bound
    \[
    \#\mathcal{S}_2 \leq \sum_{a\in [-A, A]}|s_a(K)t_a(K)-\sum_{n\leq K}c_a(n)|^{2j}\eta^{2j/10}\ll_{j, \eps} AK^{-2j+\eps}\eta^{2j/10}.
    \]
    as $A \geq K^{6j}$. Next, we note that $\prod_{p\mid a(a-4)}(1-p^{-1})= \phi(|a(a-4)|)/|a(a-4)|$. Therefore, using Lemma \ref{phi_bound} we deduce that
    \[
    \#\mathcal{S}_3\leq \sum_{a\in[-A, A]-\{0\}} \left(\frac{|a(a-4)|}{\phi(|a(a-4)|)}\right)^{18Cj}C^{18Cj}\cdot(\eta^{1/10} + A^{2\eps}K^{\eps-1/6})^{18Cj}\ll_{Cj} A.
    \]
    Adding the above three estimates we obtain the stated bound.
\end{proof}

\section{Proof of Theorem \ref{main-thm}}\label{last section}
We begin this section with a lower bound on the real density $\sigma_{\infty, a,\nu}(B,\varsigma)$ defined in \ref{real-den}. This lower bound, in conjunction with the bound on the approximate variance in Theorem \ref{variance_bound} and the small value estimate in the Theorem \ref{small_value_est} proves the main result.
\begin{Lemma}\label{real-density-lower-bound}
    Let $a \in \RR$ with $|a| \in [B^2/2, B^2]$. Then we have that
    \[
    \sigma_{\infty, a,\nu}(B,\varsigma) \gg (\log B)^2,    \]
    for all $B \geq B_0$.
\end{Lemma}
\begin{proof}
    Let us pick $0 < \eps \ll \eta \ll \delta \ll 1$. We consider the range $x \in \mathfrak{I}_1:= [B^{1-2\eta}, B^{1-\eta-3\eps}]$ and $y \in \mathfrak{I}_2:= [B^{1-\delta+\eps}, B^{1-\delta +2\eps}]$. We claim that for all suitably large integers $B$ and $x, y, a$ in the above ranges the equation
    \begin{equation}\label{zero}
         x^2+y^2+z^2-xyz=a
    \end{equation}
    has a solution $z$ with $B^{\delta+\eta+\eps}\ll |z| \ll B^{\delta+2\eta-\eps}$. It is clear that for all suitably large $B$ the quantity $x^2y^2+4(a-x^2-y^2)$ is positive, hence there exist real solutions. Let $B \geq 1$ be a sufficiently large integer. Suppose $a, x$, and $y$ are real numbers in the described range. Let $z_1, z_2$ be two real solutions for \eqref{zero} with given $a, x$, and $y$. Then we have that $\abs{z_1+z_2}=|xy|$, therefore, we can assume that $\abs{z_1}\geq \abs{xy}/2$. Further, we also have that $\abs{z_1z_2}=\abs{a-x^2-y^2}\ll B^2$, we conclude that $\abs{z_2}\ll B^{\delta+2\eta-\eps}$. Next, we note that as $z_2$ satisfies \eqref{zero} we have that $\abs{xyz_2}\gg B^2$. We conclude that $\abs{z_2} \gg B^{\delta+\eta+\eps}$. Therefore for all suitably large integers $B$, we deduce that $\nu_{B, \varsigma}(x, y, z) \gg 1$ for all $x, y$ in the described range and $z$ such that \eqref{zero} holds. We obtain that for all suitable large $B$,
    \begin{equation*}
\begin{split}
\sigma_{\infty,a,\nu}(B,\varsigma)
&= \int_{(x,y)\in \RR^2} \frac{\nu_{B,\varsigma}(x,y,z)}{\abs{2z-xy}}
\Big{\vert}_{M(x,y,z) = a}\, dx\,dy\\
&\gg \int_{(x,y)\in \mathfrak{I}_1\times \mathfrak{I}_2} \frac{1}{xy}\, dx\,dy\\
&\gg (\log B)^2.
\end{split}
\end{equation*}   
This proves the lemma.
\end{proof}
We proceed to prove the main theorem.
\begin{proof}[Proof of Theorem \ref{main-thm}]
    From the definition \eqref{var-def} of the approximate variance and Lemma \ref{zero} we deduce that 
    \[
    \frac{\#\mathcal{E}(A)}{A} \ll \frac{\mathcal{S}(A, K, \eta)}{A} + \frac{\op{Var}(B,\varsigma,K)}{\eta^2(\log B)^4A}
    \]
    for $A \in [B^2/2, B^2]$ with $B$ sufficiently large. Let $j \geq 1$ we set $K= B^{\delta_0/6j}$ to be a small power of $B$ in order to apply Theorem \ref{variance_bound} and Theorem \ref{small_value_est} with $\eta= (\log B)^{-10/j}$. We obtain that
    \[
    \frac{\#\mathcal{E}(A)}{A} \ll_j (\log B)^{-18} + \frac{(\log B)^2}{(\log B)^{4-20/j}}.
    \]
    Hence, we conclude that
    \[
    \frac{\#\mathcal{E}(A)}{A} \ll_j \frac{1}{(\log B)^{2-20/j}}.
    \]
    This completes the proof of Theorem \ref{main-thm}.
\end{proof}

\appendix

\section{Uniformity in the Kloosterman circle method}
\begin{center}
    \textsc{by Victor Y. Wang}
\end{center}

% The goal of this appendix is to prove the estimate
% \begin{equation*}
%     \sum_{\tM=0} \nu_{B,\varsigma}(x_1, y_1, a_1)\nu_{B,\varsigma}(x_2, y_2, a_2)= \sigma^{\otimes 2}_{\infty,\nu}(a_1, a_2, B,\varsigma)\sum_{m\geq 1} m^{-4}S_{\textbf{0}}(a_1, a_2; m)+O(B^{2}(\log B)^2)
% \end{equation*}
% using the delta method of \cite{duke1993bounds,heath1996new}.
% Here we must allow uniformity over the coefficients $a_1,a_2$,
% in addition to the parameters $X,Y,Z$ appearing in the weight function $\nu_{B,\varsigma}$.

Let $w\in C^\infty_c(\RR^4)$ be supported on an annulus $1\ll \norm{x}\ll 1$.
The goal of this appendix\let\thefootnote\relax\footnote{V. Y. W. was supported by the European Union's Horizon 2020 research and innovation program under the Marie Sk\l{}odowska-Curie Grant Agreement No.~101034413.}
is to prove,
uniformly over reals $\lambda_1,\dots,\lambda_4,P\ge 1$,
integers $k\ne 0$,
% and non-singular quadratic forms $F = x^t A x$ with $A$ symmetric integral
and non-singular quadratic forms $F = x^t A x$ with $A$ symmetric and $(1+1_{i\ne j})a_{ij}\in \Z$,
the following result.

\begin{theorem}
\label{blackbox}
In the setting above we have
\begin{equation}
\label{blackbox-goal}
\sum_{x\in \Z^4: F(x)=k}
w(\lambda_1x_1/P,\dots,\lambda_4x_4/P)
= \sigma_\infty \mathfrak{S}
+ O_{w,\eps}(E(F,\lambda) P^{3/2+\eps}),
\end{equation}
where
\begin{equation*}
E(F,\lambda) \defeq (1+\norm{F})^{3/4+\eps} \frac{(1+\norm{F})^{32} (\lambda_1\lambda_2\lambda_3\lambda_4)^{15}}
{\abs{\det{A}}^{8}}
+ (1+\norm{F})^{11/4+\eps} \norm{\lambda}^4.
\end{equation*}
\end{theorem}
\noindent Here and throughout the appendix, our notation is as follows:
\begin{equation*}
\begin{split}
S_q(F,k,c) &\defeq \sum_{a\in (\Z/q\Z)^\times} \sum_{x\in (\Z/q\Z)^4}
e_q(a(F(x)-k) + c\cdot x) \\
\mathfrak{S} &\defeq \sum_{q\ge 1} q^{-4} S_q(F,k,0), \\
\sigma_\infty &\defeq \lim_{\eps\to 0} (2\eps)^{-1} \int_{\abs{F(x)-k}\le \eps}
w(\lambda_1x_1/P,\dots,\lambda_4x_4/P)\, dx, \\
\norm{\lambda} &\defeq \max(\abs{\lambda_1},\dots,\abs{\lambda_4}), \\
\norm{F} &\defeq \norm{A}\defeq \max_{1\le i,j\le 4} \abs{a_{ij}}.
\end{split}
\end{equation*}

The proof of Theorem~\ref{blackbox} uses the delta method of \cite{duke1993bounds,heath1996new}.
The key feature we need is uniformity over $F$, $k$, $\lambda$, in addition to the main parameter $P$.
Such uniformity is by now standard, but we do not see how to directly obtain Theorem~\ref{blackbox} from the literature.
Therefore, we will combine various results and approaches as we see fit.
Some relevant papers, addressing uniformity and optimality in various aspects, are
\cite{dietmann2003small,
browning2008representation,
niedermowwe2010circle,
sardari2019optimal,
browning2017twisted,
browning2020density,
dymov2021refinement,
alpoge2022integers,
huang2024optimal,
kumaraswamy2024counting}.

We mainly follow \cite[\S4]{browning2020density}.
Let $Q\defeq (1+\norm{F})^{1/2} P$.
Let
\begin{equation*}
I_q(F,k,w,\lambda,P,c) \defeq \int_{\RR^4} w(\lambda_1x_1/P,\dots,\lambda_4x_4/P)
h(q/Q,(F(x)-k)/Q^2) e(-c\cdot x/q) \, dx.
\end{equation*}
The function $h(x,y)$ is defined as in \cite[\S3]{heath1996new}, as follows.
Let $\varrho_0$ be the unique function in $C^\infty_c(\RR)$
supported on $[-1,1]$
such that $\varrho_0(x)=\exp(-(1-x^2)^{-1})$ for $\abs{x}<1$.
Let $\varrho(x)\defeq 4(\int_{y\in \RR} \varrho_0(y)\,dy)^{-1}\varrho_0(4x-3)$.
Define $h: \RR_{>0} \times \RR \to \RR$ by
\begin{equation*}
h(x,y) \defeq \sum_{j\geq 1} \frac{1}{xj}
\left(\varrho(xj) - \varrho{\left(\frac{\abs{y}}{xj}\right)}\right).
\end{equation*}
Let $\Sigma(P)$ be the left-hand side of \eqref{blackbox-goal}.
By \cite[Theorem~2]{heath1996new}, since $Q>1$, we have
\begin{equation}
\label{delta-method}
\Sigma(P) = (1+O_M(Q^{-M}))
Q^{-2}\sum_{c\in \Z^4}\sum_{q\ge 1} q^{-4}S_q(F,k,c)I_q(F,k,w,\lambda,P,c),
\end{equation}
for all integers $M\ge 0$.
Moreover, if $k\ll Q^2$ and $I_q(F,k,w,\lambda,P,c)\ne 0$, then
\begin{equation}
\label{delta-method-q-cutoff}
q\ll_w Q,
\end{equation}
by the vanishing result in \cite[Lemma~4]{heath1996new}.
It remains to analyze $S_q$ and $I_q$.

For the exponential sums $S_q(F,k,c)$, we can recycle the work of \cite{browning2008representation}.

\begin{Lemma}
\label{quote-nonarch}
For any $c\in \Z^4$ and real $X\ge 1$, we have
\begin{equation*}
\sum_{q\le X} \abs{S_q(F,k,c)}
\ll_\eps \abs{\det{A}}^{1/2+\eps} \abs{k}^\eps X^{7/2+\eps}.
\end{equation*}
\end{Lemma}

\begin{proof}
This follows immediately from the $n=4$ case of \cite[Lemma~7]{browning2008representation}, because the discriminant of $F$ is simply $\det{A}$ as noted in \cite[\S2]{browning2008representation}.
But there is a minor point worth addressing.
The paper \cite{browning2008representation} technically assumes that $A$ is integral.
This is fine, because $2A$ is integral in our setting, and $\abs{S_q(F,k,c)} \le \abs{S_{2q}(2F,2k,2c)}$.
Indeed,
\begin{equation*}
S_{2q}(2F,2k,2c)
= \sum_{a\in (\Z/2q\Z)^\times} \sum_{x\in (\Z/2q\Z)^4}
e_q(a(F(x)-k) + c\cdot x)
= M_q\, S_q(F,k,c),
\end{equation*}
where $M_q \defeq \frac{\phi(2q) (2q)^4}{\phi(q) q^4} \ge 1$.
\end{proof}

For integrals $I_q(F,k,w,\lambda,P,c)$ with $c\ne 0$, we follow the treatment of \cite{browning2020density}.

\begin{Lemma}
\label{quote-arch-proof}
Suppose $k\ll Q^2$.
Let $c\in \RR^4$ be nonzero.
Then the following hold:
\begin{enumerate}
\item For any integer $M\ge 0$, we have
\begin{equation*}
I_q(F,k,w,\lambda,P,c) \ll_{w,M} \frac{1}{\lambda_1\lambda_2\lambda_3\lambda_4}
\frac{P^4Q}{q} \frac{\norm{Q\lambda}^M}{\norm{Pc}^M}.
\end{equation*}

\item We have
\begin{equation*}
I_q(F,k,w,\lambda,P,c)
\ll_w \frac{P^3q(1+\norm{F})^2\norm{\lambda}}{\abs{\det{A}}^{1/2}\norm{c}}.
\end{equation*}
\end{enumerate}
\end{Lemma}

\begin{proof}
Let $F_\lambda(x) \defeq F(x_1/\lambda_1,\dots,x_4/\lambda_4)
= x^t A_\lambda x$.
Then $A_\lambda = DAD$, where $D$ is the diagonal matrix with entries $\lambda_1^{-1},\dots,\lambda_4^{-1}$.
In particular, $\det{A_\lambda} = (\det{D})^2 \det{A}$, where $\det{D} = (\lambda_1\lambda_2\lambda_3\lambda_4)^{-1}$.
The change of variables $x\mapsto Dx$ in the definition of $I_q$ gives
\begin{equation*}
\begin{split}
\frac{I_q(F,k,w,\lambda,P,c)}{\det{D}}
&= \int_{\RR^4} w(x_1/P,\dots,x_4/P)
h(q/Q,(F_\lambda(x)-k)/Q^2) e(-b\cdot x/q) \, dx \\
&= I_q(F_\lambda,k,w,1,P,b),
\end{split}
\end{equation*}
where $b \defeq Dc = (c_1/\lambda_1,\dots,c_4/\lambda_4) \ne 0$.
To bound $I_q(F_\lambda,k,w,1,P,b)$,
we closely follow \cite[proof of Lemma~4.2]{browning2020density}.
We have
\begin{equation*}
I_q(F_\lambda,k,w,1,P,b)
= P^4 \int_{\RR^4} w(x) h(r,G(x)) e(-v\cdot x)\, dx,
\end{equation*}
where $r\defeq q/Q$, $v\defeq q^{-1}Pb$, and
\begin{equation}
\label{rescaled-G}
G(x)\defeq \frac{F(PDx)-k}{Q^2}
= \frac{F_\lambda(x)}{1+\norm{F}} - \frac{k}{Q^2}.
\end{equation}
Since $\lambda_1,\dots,\lambda_4\ge 1$, it is clear that $\norm{F_\lambda} \le \norm{F}$.
Thus $G$ and its derivatives are bounded, since $k\ll Q^2$.
Moreover, by \eqref{delta-method-q-cutoff}, we may assume $r\ll 1$.
Repeated integration by parts over $x$, as in \cite[proof of Lemma~4.2(i)]{browning2020density}, gives
\begin{equation*}
I_q(F_\lambda,k,w,1,P,b)
\ll_{w,M} \frac{r^{-1-M}}{\norm{v}^M}
= \frac{Q/q}{\norm{Pb/Q}^M}.
\end{equation*}
Since $\norm{b} \ge \norm{c}/\norm{\lambda}$, the bound in (1) follows.

Next, arguing as in \cite[proof of Lemma~4.2(ii)]{browning2020density}, we have
\begin{equation*}
I_q(F_\lambda,k,w,1,P,b) = P^4 r^{-1}
\int_{\RR} p(t) \int_{\RR^4} w_1(x) e(tG(x) - v\cdot x)\, dx\, dt,
\end{equation*}
where $p$, $w_1$ are defined as follows.
Fix $K>0$ such that $\abs{G(x)}\le K$ holds
for all choices of $x\in \operatorname{supp}(w)$ and $k\ll Q^2$.
Fix $\upsilon\in C^\infty_c(\RR)$ such that $\upsilon\vert_{[-2K,2K]} \ge 1$.
Finally, let $w_1(x) \defeq w(x)/\upsilon(G(x))$
and $p(t) \defeq r\int_{\xi\in \RR} \upsilon(\xi) h(r,\xi)e(-t\xi)\, d\xi$.
Here $p(t)\ll r$.
% Note: More generally, $p(t) \ll_M r(r\abs{t})^{-M}$, as noted in \cite[paragraph before Lemma~17]{heath1996new}; but we only seem to need the case $M=0$ here.

Conveniently, the absolute value $\abs{J}$ of the integral
\begin{equation*}
J\defeq \int_{\RR^4} w_1(x) e(tG(x) - v\cdot x)\, dx
\end{equation*}
does not depend on $k$.
Therefore, the bounds on $J$ in \cite[proof of Lemma~4.2(ii)]{browning2020density} for $k=0$ extend to all $k\ll Q^2$.
This gives $J\ll_{w,M} \norm{v}^{-M}$ if $\norm{v}\gg \abs{t}$, and
\begin{equation*}
J\ll_w \frac{1}{\abs{\det(t(1+\norm{F})^{-1}A_\lambda)}^{1/2}}
= \frac{(1+\norm{F})^2}{t^2 (\det{D}) \abs{\det{A}}^{1/2}}.
\end{equation*}
otherwise.
Thus
\begin{equation*}
\frac{I_q(F_\lambda,k,w,1,P,b)}{P^4} \ll_w
\int_{\abs{t}\ll \norm{v}} \norm{v}^{-2} \, dt
+ \int_{\abs{t}\gg \norm{v}} \frac{(1+\norm{F})^2}{t^2 (\det{D}) \abs{\det{A}}^{1/2}} \, dt.
\end{equation*}
Since $(\det{D}) \abs{\det{A}}^{1/2} \le \abs{\det{A}}^{1/2} \ll (1+\norm{F})^2$, it follows that
\begin{equation*}
\frac{I_q(F,k,w,\lambda,P,c)}{\det{D}}
= I_q(F_\lambda,k,w,1,P,b)
\ll \frac{P^4 (1+\norm{F})^2}{\norm{v} (\det{D}) \abs{\det{A}}^{1/2}}.
\end{equation*}
Since $\norm{v} \ge Pq^{-1}\norm{c}/\norm{\lambda}$, the bound in (2) follows.
\end{proof}

For $c=0$, the simplest approach is just to make \cite[Lemma~13]{heath1996new} effective.

\begin{Lemma}
\label{effective-c=0}
Suppose $k\ll Q^2$ and $q\ll Q$.
Then we have
\begin{equation*}
\frac{I_q(F,k,w,\lambda,P,0)}{Q^2}
= \sigma_\infty
+ O_{w,M}{\left(\frac{P^4}{Q^2} \frac{(q/Q)^M
(1+\norm{F})^{16M+16} (\lambda_1\lambda_2\lambda_3\lambda_4)^{8M+7}}
{\abs{\det{A}}^{4M+4}}\right)}
\end{equation*}
for all $M\in \Z_{\ge 0}$.
Moreover, we have
\begin{equation*}
\sigma_\infty \ll_w \frac{P^4}{Q^2}
\frac{(1+\norm{F})^4 \lambda_1\lambda_2\lambda_3\lambda_4}{\abs{\det{A}}}.
\end{equation*}
\end{Lemma}

\begin{proof}
We begin with the same change of variables $x\mapsto Dx$ as in the proof of Lemma~\ref{quote-arch-proof}.
Since the adjugate of $A_\lambda$ is a $4\times 4$ matrix with cubic polynomial entries, we have
$$\norm{A_\lambda^{-1}}
\ll \frac{\norm{A_\lambda}^3}{\abs{\det{A_\lambda}}}
\le \frac{\norm{A}^3}{(\det{D})^2\abs{\det{A}}}.$$
Therefore,
\begin{equation*}
\norm{\nabla{G}(x)}
= \frac{\norm{2A_\lambda x}}{(1+\norm{F})}
\gg \frac{\norm{x}}{(1+\norm{F})\norm{A_\lambda^{-1}}}
\gg_w \frac{(\det{D})^2\abs{\det{A}}}{(1+\norm{F})^4}
\end{equation*}
for all $x\in \operatorname{supp}(w)$.
Fix a smooth partition of unity $1 = \omega_0+\omega_1+\dots+\omega_4$ of $\RR^4$ such that $\omega_0$ is supported on $\norm{x}\le 2$,
and $\omega_i$ is supported on $\abs{x_i}\ge 1$ for each $i\ge 1$.
Then there exists a real $C_{F,\lambda} \asymp_w (\det{D})^2\abs{\det{A}}/(1+\norm{F})^4$ such that
\begin{equation*}
I_q(F_\lambda,k,w,1,P,0)
= P^4 \sum_{1\le i\le 4} J_i,
\end{equation*}
where
\begin{equation*}
J_i\defeq \int_{\RR^4} \omega_i(\nabla{G}(x)/C_{F,\lambda}) w(x) h(r,G(x))\, dx.
\end{equation*}
By symmetry, we may concentrate on $J_4$.
Let
\begin{equation*}
f(x_1,x_2,x_3,y) \defeq
\frac{\omega_4(\nabla{G}(x)/C_{F,\lambda}) w(x)\vert_{G(x)=y}}{\partial_{x_4}{G}(x)},
\end{equation*}
where $\vert_{G(x)=y}$ denotes summation over all solutions $x_4\in \operatorname{supp}(w(x_1,x_2,x_3,\cdot))$ to the equation $G(x)=y$.
By the implicit function theorem, we have
\begin{equation*}
J_4 = \int_{\RR^3} \int_{\RR} f(x_1,x_2,x_3,y) h(r,y)\, dy\,dx_1\,dx_2\,dx_3,
\end{equation*}
We now seek to apply an effective version of \cite[Lemma~9]{heath1996new} to estimate the inner integral over $y$ in $J_4$.
For this, observe that $\partial_y(x_4) = 1/\partial_{x_4}{G}(x)$.
So
\begin{equation*}
\partial_y^n f(x_1,x_2,x_3,y)
\ll_{w,n} \frac{1}{C_{F,\lambda}^{2n+1}}
\end{equation*}
for all $n\ge 0$,
by the product and chain rules.
Therefore, Lemma~\ref{taylor} implies
\begin{equation*}
J_4 = \int_{\RR^3} f(x_1,x_2,x_3,0) \,dx_1\,dx_2\,dx_3
+ O_{w,M}(r^M/C_{F,\lambda}^{4M+4})
\end{equation*}
for all integers $M\ge 0$.
Using \cite[\S5.4, par.~4]{chambert2010igusa} to pass to a real density, we get
\begin{equation*}
\sum_{1\le i\le 4} J_i
= \lim_{\eps\to 0} (2\eps)^{-1}
\int_{\abs{G(x)}\le \eps} w(x)\, dx
+ O_{w,M}(r^M/C_{F,\lambda}^{4M+4}).
\end{equation*}
Moreover, by definition,
\begin{equation*}
\sigma_\infty
= \lim_{\eps\to 0} (2\eps)^{-1}
\int_{\abs{F(x)-k}\le \eps} w(D^{-1}x/P)\, dx
= \frac{\det(PD)}{Q^2}\lim_{\eps\to 0} (2\eps)^{-1}
\int_{\abs{G(x)}\le \eps} w(x)\, dx,
\end{equation*}
since $F(x)-k = Q^2G(D^{-1}x/P)$ by \eqref{rescaled-G}.
Finally,
\begin{equation*}
\frac{I_q(F,k,w,\lambda,P,0)}{Q^2}
= \frac{(\det{D})P^4}{Q^2} \sum_{1\le i\le 4} J_i
= \sigma_\infty
+ O_{w,M}{\left(\frac{(\det{D})P^4}{Q^2} \frac{r^M}{C_{F,\lambda}^{4M+4}}\right)},
\end{equation*}
which immediately translates into the desired estimate for $I_q$.
Moreover, clearly
\begin{equation*}
\sigma_\infty
\ll_w \frac{\det(PD)}{Q^2} \frac{1}{C_{F,\lambda}},
\end{equation*}
since $f(x_1,x_2,x_3,0) \ll 1/C_{F,\lambda}$.
\end{proof}

\begin{Lemma}
\label{taylor}
Let $K,T\ge 1$.
Suppose $f\in C^\infty(\RR)$ is supported on $[-K,K]$
and that $f^{(n)}(y) \ll_n T^{n+1}$ for all $n\in \Z_{\ge 0}$.
Suppose $0<x\ll 1$ and $M\in \Z_{\ge 0}$.
Then
\begin{equation}
\label{taylor-goal}
\mathcal{I}(f)
\defeq \int_{\RR} f(y)h(x,y)\, dy
= f(0) + O_M(KT^{2M+2}x^M).
\end{equation}
\end{Lemma}

\begin{proof}
By Taylor's theorem, we have
\begin{equation*}
f(y) = \sum_{0\le n<N} \frac{f^{(n)}(0)}{n!} y^n
+ O_N(T^{N+1}\abs{y}^N).
\end{equation*}
Moreover, $h(x,y)\ll x^{-1}$ by \cite[Lemma~4]{heath1996new},
and for $\abs{y}\ge x^{1/2}$ we have $h(x,y)\ll_N x^N$ by \cite[Lemma~5]{heath1996new}.
Therefore
\begin{equation*}
\mathcal{I}(f)
= \mathcal{I}_0(f)
+ \int_{\abs{y}\le x^{1/2}} O_N(T^{N+1}\abs{y}^N x^{-1})\, dy
+ \int_{x^{1/2}\le \abs{y}\le K} O_N(Tx^N)\, dy,
\end{equation*}
where
\begin{equation*}
\begin{split}
\mathcal{I}_0(f) &\defeq
\int_{\abs{y}\le x^{1/2}} \sum_{0\le n<N} \frac{f^{(n)}(0)}{n!} y^n h(x,y)\, dy \\
&= f(0) (1 + O_N(x^{N-1/2} + x^{N/2}))
+ \sum_{1\le n<N} O_n(T^{n+1} x^{n/2} (x^{N-1/2} + x^{N/2})),
\end{split}
\end{equation*}
by \cite[Lemmas~6 and~8]{heath1996new}.
Taking $N=2M+1$, we get \eqref{taylor-goal}, since $x\ll 1\le T,K$.
\end{proof}

We now combine the ingredients above.
First note that $\Sigma(P) = \sigma_\infty = 0$ unless $k\ll_w Q^2$.
Therefore, we may and do assume $k\ll Q^2$.
By Lemma~\ref{quote-arch-proof}(1), we have
\begin{equation*}
\sum_{\norm{c}\ge Q^\eps \norm{Q\lambda/P}}
\sum_{q\ll Q} q^{-4}\abs{S_q(c)I_q(c)}
\ll_M Q^{-M},
\end{equation*}
where we abbreviate $S_q(F,k,c)$ to $S_q(c)$ and $I_q(F,k,w,\lambda,P,c)$ to $I_q(c)$.
By Lemma~\ref{quote-nonarch} and Lemma~\ref{quote-arch-proof}(2), applied in dyadic intervals $q\in [2^j,2^{j+1})$ with $j\in \Z_{\ge 0}$, we have
\begin{equation*}
\sum_{q\ll Q} q^{-4}\abs{S_q(c)I_q(c)}
\ll_{w,\eps} Q^\eps \max_{1\ll X\ll Q}
\frac{\abs{\det{A}}^{1/2+\eps} \abs{k}^\eps}{X^{1/2-\eps}}
\frac{P^3X(1+\norm{F})^2\norm{\lambda}}{\abs{\det{A}}^{1/2}\norm{c}}.
\end{equation*}
Thus
\begin{equation*}
\sum_{\norm{c}\le Q^\eps \norm{Q\lambda/P}}
\sum_{q\ll Q} q^{-4}\abs{S_q(c)I_q(c)}
\ll_{w,\eps} Q^{1/2+5\eps} \norm{Q\lambda/P}^3
\abs{\det{A}}^{\eps} \abs{k}^\eps
P^3(1+\norm{F})^2\norm{\lambda}.
\end{equation*}
For $c=0$, we first use Lemma~\ref{effective-c=0} with $M=1$,
then use Lemma~\ref{quote-nonarch}, to get
\begin{equation*}
\begin{split}
\sum_{q\ll Q} q^{-4}S_q(0)\frac{I_q(0)}{Q^2}
= \sigma_\infty \mathfrak{S}
&+ O_{w,\eps}{\left(\frac{\abs{\det{A}}^{1/2+\eps} \abs{k}^\eps}{Q^{1/2-\eps}}
\frac{P^4}{Q^2}\frac{(1+\norm{F})^4 \lambda_1\lambda_2\lambda_3\lambda_4}{\abs{\det{A}}}\right)} \\
&+ O_{w,\eps}{\left(\frac{\abs{\det{A}}^{1/2+\eps} \abs{k}^\eps}{Q^{1/2-\eps}}
\frac{P^4}{Q^2} \frac{(1+\norm{F})^{32} (\lambda_1\lambda_2\lambda_3\lambda_4)^{15}}
{\abs{\det{A}}^{8}}\right)}.
\end{split}
\end{equation*}
By \eqref{delta-method}, \eqref{delta-method-q-cutoff},
and the bounds $k\ll Q^2$ and $\det{A} \ll \norm{F}^4$,
we conclude that
\begin{equation*}
\frac{\Sigma(P) - \sigma_\infty \mathfrak{S}}{(1+\norm{F})^{2+4\eps}}
\ll_{w,\eps}
\frac{P^4}{Q^{5/2-3\eps}} \frac{(1+\norm{F})^{32} (\lambda_1\lambda_2\lambda_3\lambda_4)^{15}}
{\abs{\det{A}}^{8}}
+ Q^{3/2+7\eps}\norm{\lambda}^4.
\end{equation*}
Since $Q = (1+\norm{F})^{1/2} P$, this implies Theorem~\ref{blackbox}.

\bibliographystyle{alpha}
\bibliography{references}
\end{document}